\documentclass[a4paper,10pt]{amsart}

\usepackage[T1]{fontenc}
\usepackage{quiver}
\usepackage{tikz-cd}
\usepackage{amsfonts}
\usepackage{amssymb}
\usepackage{amsthm}
\usepackage{amsmath}
\usepackage{hyperref}
\usepackage{geometry}

\theoremstyle{plain}
\newtheorem{theorem}{Theorem}[section]

\theoremstyle{plain}
\newtheorem{corollary}[theorem]{Corollary}

\theoremstyle{plain}
\newtheorem{lemma}[theorem]{Lemma}

\theoremstyle{plain}
\newtheorem{proposition}[theorem]{Proposition}

\theoremstyle{plain}
\newtheorem{conjecture}[theorem]{Conjecture}

\theoremstyle{definition}
\newtheorem{definition}[theorem]{Definition}

\theoremstyle{definition}
\newtheorem{example}[theorem]{Example}

\theoremstyle{remark}
\newtheorem{remark}[theorem]{Remark}

\begin{document}

\title{A modern perspective on rational homotopy theory}
\author{Eleftherios Chatzitheodoridis}
\address{Department of Mathematics, University of Virginia, 141 Cabell Drive, Kerchof Hall, Charlottesville, VA 22904, USA}
\email{thp5uc@virginia.edu}
\maketitle

\begin{abstract}
In Quillen's paper on rational homotopy theory, the category of $1$-reduced simplicial sets is endowed with a family of model structures, the most prominent of which is the one in which the weak equivalences are the rational homotopy equivalences and the fibrant objects are the rational Kan complexes. In this paper, we give a modern approach to this family of model structures. We recover Quillen's family of model structures by first left-transferring the model structure on pointed simplicial sets and then left Bousfield localizing at the rationalization maps of spheres. Applying this localization to the model category of all spaces yields a model category in which the weak equivalences are the rational homotopy equivalences in the extended sense of G{\'{o}}mez-Tato, Halperin, and Tanr{\'{e}} and the fibrant objects are the rational spaces. Thus, we generalize Quillen's family of model structures beyond the rational homotopy theory of $1$-connected spaces.
\end{abstract}

\section{Introduction}\label{section1}
Classically, rational homotopy theory is the study of $1$-connected spaces up to the torsion in their homotopy groups. The non-vanishing homotopy groups of a $1$-connected space are abelian, so their torsion can be discarded by tensoring with the rational numbers. In this framework, we adjust our notion of equivalence from that of a weak homotopy equivalence, which preserves all homotopy groups on the nose, to that of a \emph{rational homotopy equivalence}, which induces isomorphisms of higher homotopy groups only after tensoring with the rational numbers. Conceptually, a rational homotopy equivalence can be thought of as a weak homotopy equivalence up to an error in torsion. In the $1$-connected setting, one can equivalently use rational homology instead; a map of $1$-connected spaces is a rational homotopy equivalence if and only if it is a rational homology isomorphism.

Accordingly, the corresponding notion of a nice space in rational homotopy theory is a space whose higher homotopy groups are already rational vector spaces, that is, uniquely divisible by every positive integer; such spaces are called \emph{rational}. In \cite{Sullivan1970}, Sullivan constructs, for a $1$-connected space $X$, a relative CW complex $(X_{\mathbb{Q}}, X)$ such that $X_{\mathbb{Q}}$ is a rational space and the relative CW inclusion is a rational homotopy equivalence. This construction is called the \emph{rationalization} of $X$, and it identifies the rational homotopy groups of $X$ with the higher homotopy groups of $X_{\mathbb{Q}}$ in a topological fashion.

The geometric construction that underpins the rationalization of spaces is the rationalization of the $n$-sphere $S^{n}$ in dimensions $n \geq 2$ and, more specifically, the construction of the \emph{rational $n$-sphere} $S^{n}_{\mathbb{Q}}$. In essence, the rationalization of a general space is built out of the CW inclusions of $n$-spheres into rational $n$-spheres, hence the resulting relative CW construction. Thus, the rationalization maps of spheres are the building blocks for the rational homotopy theory of spaces.

Working rationally can simplify computations significantly. In the case of spheres, though their homotopy groups are notoriously difficult to compute, their rational homotopy groups were computed by Serre in \cite{Serre1951}. This contrast captures the fact that the difficulty in computing the homotopy groups of spheres is concentrated in their torsion. It also suggests that, while a direct computation of homotopy groups can be too ambitious an approach, the computation of localizations of homotopy groups can be a more tractable goal.

\subsection{Quillen's approach to rational homotopy theory}\label{section1.1} The computational features of rational homotopy theory extend deeper than the systematic disregard of torsion in homotopy groups. Rational homotopy theory can be translated purely into algebraic terms, and rational computations in geometry and topology can be transformed into more feasible ones in algebra. Quillen's paper on rational homotopy theory \cite{Quillen1969} establishes this advantage of rational homotopy theory, and it informs us about the algebra that fully encodes the rational homotopy theory of $1$-connected spaces. Specifically, Quillen's work turns the study of the rational homotopy theory of $1$-connected spaces into that of rational differential graded Lie algebras and rational differential graded coalgebras up to quasi-isomorphism.

Quillen's approach to rational homotopy theory takes place within his framework of \emph{model categories} from \cite{Quillen1967}. A model structure on a category allows one to introduce tools and constructions from homotopy theory in that category, and thus endows the category with a homotopy theory. Specifically, a model structure on a category comprises a class of morphisms called \emph{weak equivalences} together with two other classes of morphisms, \emph{cofibrations} and \emph{fibrations}, subject to a list of axioms. This structure allows for constructions such as the homotopy category of a model category, which is its localization at the class of weak equivalences. In practice, model categories often satisfy additional properties that render the study of their homotopy theory less abstract and more explicit.

In his paper on rational homotopy theory, Quillen produces a number of model structures, including

\begin{enumerate}

\item one on $1$-connected spaces whose weak equivalences are the rational homotopy equivalences;

\item one on rational differential graded Lie algebras whose weak equivalences are the quasi-isomorphisms; and

\item one on rational differential graded coalgebras whose weak equivalences are also the quasi-isomorphisms,

\end{enumerate}
as well as functorial constructions that induce equivalences of these model categories. The notion of equivalence that we are referring to is, in fact, called a \emph{Quillen equivalence}. Thus, Quillen's work quantifies how rational homotopy theory translates difficult problems in geometry and topology into more feasible problems in algebra. In other words, Quillen uses his framework of model categories to model the rational homotopy theory of $1$-connected spaces with the homotopy theory of rational differential graded Lie algebras and coalgebras.

On the side of $1$-connected spaces, Quillen does not work so much with topological spaces as with \emph{simplicial sets}. Simplicial sets can be thought of as simplicial complexes with degeneracy data that allows us to view lower-dimensional simplices as honorary higher-dimensional simplices, as well as directionality data for all simplices. This point of view illustrates that simplicial sets are a combinatorial model for spaces, and the degeneracy and directionality data render them a more faithful model than that of ordinary simplicial complexes. An exposition of this approach to simplicial sets and their advantages over simplicial complexes is provided by Bergner in \cite{Bergner2022}.

Instead of working directly with $1$-connected topological spaces, Quillen produces a family of model structures on the category of \emph{$1$-reduced simplicial sets}, that is, simplicial sets whose $1$-skeleton is a point. As the category of $1$-connected topological spaces is ill-behaved with respect to important limits and colimits, such as pullbacks, the passage to $1$-reduced simplicial sets is necessary for Quillen's model-categorical approach, which requires a category that has all finite limits and colimits. Thus, Quillen uses $1$-reduced simplicial sets as a combinatorial analog of $1$-connected spaces in analogy with how simplicial sets are a combinatorial model for spaces, except that, now, the more rigid structure of a $1$-reduced simplicial set yields a nicer underlying category to work in than that of $1$-connected spaces.

More specifically, Quillen produces a family of model structures on $1$-reduced simplicial sets that interpolates between ordinary homotopy theory, whose weak equivalences are the weak homotopy equivalences, and rational homotopy theory, whose weak equivalences are the rational homotopy equivalences. This interpolation captures the fact that we have options between discarding no torsion at all, as is the case with weak homotopy equivalences, and discarding all torsion, as is the case with rational homotopy equivalences. For instance, given a prime number $p$, we can specifically preserve $p$-torsion in higher homotopy groups by tensoring them with the $p$-local integers $\mathbb{Z}_{(p)}$, and we can adapt our notion of equivalence to be that of a \emph{$p$-local homotopy equivalence}, which induces isomorphisms of higher homotopy groups only after tensoring with the $p$-local integers. Such a selective approach allows us to study the torsion in homotopy groups one prime at a time rather than study it in one go or discard it altogether.

In full detail, for a multiplicative subset $M$ of $\mathbb{Z}$, Quillen endows the category of $1$-reduced simplicial sets with a model structure in which a map is a weak equivalence if and only if it induces isomorphisms of higher homotopy groups after tensoring with $M^{-1}\mathbb{Z}$; we call such a map an \emph{$M$-local homotopy equivalence}. Tensoring with $M^{-1}\mathbb{Z}$ for different choices of $M$ allows us to preserve specific types of torsion in homotopy groups. For instance, for a prime number $p$, choosing $M=\mathbb{Z} \setminus p\mathbb{Z}$ and tensoring with $\mathbb{Z}_{(p)}$ only keeps $p$-torsion.

In this notation, a space is called \emph{$M$-local} if its higher homotopy groups are $M^{-1}\mathbb{Z}$-modules, that is, uniquely divisible by every positive integer in $M$. Moreover, for a $1$-connected space $X$, Sullivan's geometric construction yields a relative CW complex $(X_{M^{-1}\mathbb{Z}}, X)$ such that $X_{M^{-1}\mathbb{Z}}$ is an $M$-local space and the relative CW inclusion is an $M$-local homotopy equivalence. This construction is called the \emph{$M$-localization} of $X$. Thus, rational homotopy theory can be developed for any localization of $\mathbb{Z}$, though specific computations, such as Serre's computation of the rational homotopy groups of spheres, much depend on the choice of such a localization.

The following table depicts the interpolation given by varying our choice of a multiplicative subset $M$ of $\mathbb{Z}$.

\begin{table}[h]
\begin{tabular}{|c|c|c|c|}
\hline
{\color{blue}{$M$}}                                                                       & $\{1\}$      & $\mathbb{Z} \setminus p\mathbb{Z}$ & $\mathbb{Z} \setminus \{0\}$ \\ \hline
{\color{blue}{$M^{-1}\mathbb{Z}$}}                                                        & $\mathbb{Z}$ & $\mathbb{Z}_{(p)}$                 & $\mathbb{Q}$                 \\ \hline
{\color{blue}{torsion discarded by $- \otimes M^{-1}\mathbb{Z}$}} & none         & all but $p$-torsion                & all                          \\ \hline
{\color{blue}{$M$-local homotopy equivalence}}                                                         & weak         & $p$-local                          & rational                     \\ \hline
\end{tabular}
\label{table1}
\end{table}

\subsection{A modern perspective on Quillen's approach}\label{section1.2} To establish his family of model structures on $1$-reduced simplicial sets, Quillen painstakingly verifies the model category axioms one-by-one. The motivating idea behind our work in this paper is that we can use modern tools to provide a streamlined approach to this family of model structures. The first tool we employ is that of left transfer, as developed by Bayeh, Hess, Karpova, K{\k{e}}dziorek, Riehl, and Shipley in \cite{BHKKRS2015}. This tool allows us to transfer the model structure on pointed simplicial sets to a model structure on $1$-reduced simplicial sets whose weak equivalences are the weak homotopy equivalences. Then, we left Bousfield localize the latter model structure at the $M$-localization maps of spheres to recover the rest of Quillen's family. Although the tool of left Bousfield localization dates back to the work of Bousfield \cite{Bousfield1975, Bousfield1977} and Farjoun \cite{Farjoun1992, Farjoun1995}, we rely on the more recent existence results for left Bousfield localization developed by Hirschhorn in \cite{Hirschhorn2003} and by Barwick in \cite{Barwick2010}.

For the left transfer, we observe that a $1$-reduced simplicial set has a unique $0$-simplex, so there is a forgetful functor from the category of $1$-reduced simplicial sets to the category of pointed simplicial sets. This functor admits a right adjoint given by the \emph{$1$st Eilenberg subcomplex} of a pointed simplicial set, as defined by May in \cite{May1967}; in fact, Quillen uses this construction to give his family of model structures. We use the $1$st Eilenberg subcomplex functor to left-transfer the model structure on pointed simplicial sets to a model structure on $1$-reduced simplicial sets whose weak equivalences are the weak homotopy equivalences and whose fibrant objects are the Kan complexes.

Then, for the left Bousfield localization, we characterize, for a multiplicative subset $M$ of $\mathbb{Z}$, the $M$-local spaces and the $M$-local homotopy equivalences using function complexes. Our characterizations can be viewed as the enriched-in-simplicial-sets version of the work of Casacuberta, Peschke, and Pfenniger in \cite{CPP1992}, which involves hom-sets instead. Using these characterizations, we left Bousfield localize our left-induced model category at the $M$-localization maps of spheres to obtain a model structure on $1$-reduced simplicial sets whose weak equivalences are the $M$-local homotopy equivalences and whose fibrant objects are the $M$-local Kan complexes. Thus, we recover the entirety of Quillen's family of model structures on $1$-reduced simplicial sets.

In particular, we obtain a model structure on the category of $1$-reduced simplicial sets in which the weak equivalences are the rational homotopy equivalences and the fibrant objects are the rational Kan complexes. This model category encapsulates the rational homotopy theory of $1$-connected spaces; for instance, fibrant replacements in this model category are rationalizations of $1$-connected spaces. It is this model category that Quillen establishes to be equivalent to his model structures on rational differential graded Lie algebras and rational differential graded coalgebras.

We also give an alternative approach using Bousfield localization with respect to homology, as introduced by Bousfield in \cite{Bousfield1975}, in conjunction with left transfer. The idea is that, for a multiplicative subset $M$ of $\mathbb{Z}$, rather than first left-transfer the model structure on pointed simplicial sets and then localize, we can try to localize the model structure on pointed simplicial sets with respect to homology with $M^{-1}\mathbb{Z}$-coefficients and then left-transfer this localized model structure. We show that this approach also recovers Quillen's family of model structures on $1$-reduced simplicial sets.

Schematically, for a multiplicative subset $M$ of $\mathbb{Z}$, we denote

\begin{enumerate}

\item the model category of pointed simplicial sets by ${{\mathcal{SS}\mathrm{ets}}_{\ast}}$;

\item the model structure on $1$-reduced simplicial sets with weak homotopy equivalences by ${\mathcal{SS}\mathrm{ets}_{1}}$;

\item the localization of ${{\mathcal{SS}\mathrm{ets}}_{\ast}}$ with respect to homology with $M^{-1}\mathbb{Z}$-coefficients by ${\mathcal{L}_{HM^{-1}\mathbb{Z}}{\mathcal{SS}\mathrm{ets}}_{\ast}}$; and

\item the model structure on $1$-reduced simplicial sets with $M$-local homotopy equivalences by ${\mathcal{L}_{M^{-1}\mathbb{Z}}\mathcal{SS}\mathrm{ets}_{1}}$,

\end{enumerate}
and we show that the diagram of model-categorical operations

\[\begin{tikzcd}[row sep=6mm, column sep=6mm, ampersand replacement=\&]
	{{\mathcal{SS}\mathrm{ets}}_{\ast}} \&\& {\mathcal{SS}\mathrm{ets}_{1}} \\
	\\
	{\mathcal{L}_{HM^{-1}\mathbb{Z}}{\mathcal{SS}\mathrm{ets}}_{\ast}} \&\& {\mathcal{L}_{M^{-1}\mathbb{Z}}\mathcal{SS}\mathrm{ets}_{1}}
	\arrow["{\textrm{localization}}"', squiggly, from=1-1, to=3-1]
	\arrow["{\textrm{transfer}}", squiggly, from=1-1, to=1-3]
	\arrow["{\textrm{localization}}", squiggly, from=1-3, to=3-3]
	\arrow["{\textrm{transfer}}", squiggly, from=3-1, to=3-3]
\end{tikzcd}\]
commutes. As the weak equivalences in ${\mathcal{L}_{HM^{-1}\mathbb{Z}}{\mathcal{SS}\mathrm{ets}}_{\ast}}$ are the $M$-local homology isomorphisms, the commutativity of the diagram stems from the result that a map of $1$-connected spaces is an $M$-local homotopy equivalence if and only if it is an $M$-local homology isomorphism.

Another motivating goal of ours is to show that these model structures on $1$-reduced simplicial sets have properties that one generally desires and expects in the context of model categories. The tools that we invoke guarantee us most of these properties. An exception is the property of being a simplicial model category, as defined by Quillen in \cite{Quillen1967}. We establish this property explicitly by first showing that our left-induced model structure on $1$-reduced simplicial sets is simplicial. As left Bousfield localization preserves the simplicial structure, this structure is bestowed upon the entire family of model structures on $1$-reduced simplicial sets.

\subsection{Generalizing Quillen's family of model structures to all spaces}\label{section1.3} The modern tools that we use to recover Quillen's family of model structures on $1$-reduced simplicial sets allow us to extend it to all spaces. This extension encodes the generalization of rational homotopy theory to non-simply connected spaces, which has been accomplished in two stages. Firstly, rational homotopy theory was extended to nilpotent spaces by Sullivan \cite{Sullivan1974, Sullivan1977, Sullivan1970}; Bousfield and Kan \cite{BousfieldKan1972}; Hilton, Mislin, and Roitberg \cite{HMR1975}; Bousfield and Gugenheim \cite{BousfieldGugenheim1976}; and Neisendorfer \cite{Neisendorfer1978}. Ultimately, it was extended to all spaces by Casacuberta and Peschke \cite{CasacubertaPeschke1993}; G{\'{o}}mez-Tato, Halperin, and Tanr{\'{e}} \cite{GTHT2000}; and Bastardas and Casacuberta \cite{BastardasCasacuberta2001}. At both stages of generalization, rational homotopy theory has been translated to purely algebraic terms, as in the classical case.

We use our modern perspective to produce a family of model structures on the category of spaces that describes the generalization of rational homotopy theory to non-simply connected spaces due to G{\'{o}}mez-Tato, Halperin, and Tanr{\'{e}}. In their approach, one rationalizes a connected space $X$ by applying fiberwise rationalization to the fibration \[\widetilde{X} \rightarrow X \rightarrow K(\pi_{1}(X), 1)\] given by the $1$st Postnikov approximation of $X$, where $K(\pi_{1}(X), 1)$ is an Eilenberg-Mac Lane space. The fiber $\widetilde{X}$ is the universal cover of $X$, which is a $1$-connected space that can be rationalized in the classical sense. This technique generalizes to disconnected spaces by fiberwise rationalizing every component. The tool of fiberwise localization is invoked by Sullivan \cite{Sullivan1974} and investigated by Llerena \cite{Llerena1982, Llerena1985} and Farjoun \cite{Farjoun1995_CI}, and the specific case of fiberwise rationalization is treated in recent work of Ivanov \cite{Ivanov2022}.

In general, for a multiplicative subset $M$ of $\mathbb{Z}$, the fiberwise $M$-localization of a connected space $X$ entails, at the level of total spaces, a map $X \rightarrow X_{M^{-1}\mathbb{Z}}$ that induces isomorphisms on $\pi_{0}$ and $\pi_{1}$ and isomorphisms of higher homotopy groups after tensoring with $M^{-1}\mathbb{Z}$. We call such a map an \emph{$M$-local homotopy equivalence}, for it extends the notion of an $M$-local homotopy equivalence beyond maps of $1$-connected spaces. In \cite{RWZ2021}, Rivera, Wierstra, and Zeinalian give algebraic characterizations for rational homotopy equivalences in this general sense, which they refer to as \emph{$\pi_{1}$-rational homotopy equivalences}. Also, as in the $1$-connected case, a space is called \emph{$M$-local} if its higher homotopy groups are $M^{-1}\mathbb{Z}$-modules; note that no conditions are imposed on $\pi_{0}$ or $\pi_{1}$.

In this terminology, we endow the category of spaces with a model structure whose weak equivalences are the $M$-local homotopy equivalences and whose fibrant objects are the $M$-local spaces. This model category is the left Bousfield localization of the model category of spaces at the $M$-localization maps of spheres. In particular, we get a model structure on the category of spaces in which the weak equivalences are the rational homotopy equivalences and the fibrant objects are the rational spaces. In \cite{GTHT2000}, G{\'{o}}mez-Tato, Halperin, and Tanr{\'{e}} produce an algebraic category with a notion of homotopy between its morphisms whose homotopy category is equivalent to the homotopy category of rational spaces in their extended sense. By our theorem, the latter homotopy category arises from a model category, so it can be studied using the structured toolkit of model categories.

Lastly, we note that our localization of all spaces at rational homotopy does not coincide with Bousfield's localization of spaces at rational homology; the unique map from $\mathbb{R}P^{2}$ to the point is a rational homology isomorphism, but not a rational homotopy equivalence in the extended sense, for it does not induce an isomorphism of fundamental groups. This result captures the fact that rational homotopy and rational homology do not agree for non-simply connected spaces.

\subsection{Organization of the paper}\label{section1.4} Section \ref{section2} comprises background on $1$-reduced simplicial sets. In Section \ref{section3}, we perform the left transfer. In Section \ref{section4}, we investigate the fibrations in our left-induced model category. In Section \ref{section5}, we establish that our left-induced model category is simplicial. Section \ref{section6} consists of auxiliary characterization results from the rational homotopy theory of $1$-connected spaces. We use these results in Section \ref{section7} to perform the left Bousfield localization that recovers Quillen's family of model structures on $1$-reduced simplicial sets. In Section \ref{section8}, we give an alternative approach via Bousfield localization with respect to homology. In Section \ref{section9}, we give an overview of the rational homotopy theory of non-simply connected spaces. Lastly, in Section \ref{section10}, we perform our left Bousfield localization to all spaces to generalize Quillen's family of model structures.

\subsection{Acknowledgments}\label{section1.5} I am grateful to my PhD advisor at the University of Virginia, Julie Bergner, for her guidance, input, feedback, and support that made the work in this paper possible. I would also like to thank Nicholas Kuhn for suggesting the approach of using Bousfield localization with respect to homology, as well as Scott Balchin and Kimball Strong for enlightening conversations about this work.

\section{$1$-reduced simplicial sets}\label{section2}
We begin with some background on $1$-reduced simplicial sets. We first give background on simplicial sets; in particular, we define the $1$-skeleton and the $1$-coskeleton of a simplicial set. Using these notions, we define $1$-reduced simplicial sets. Then, we define the {$1$-reduction} of a simplicial set, as well as the {$1$st Eilenberg subcomplex} of a pointed simplicial set. These constructions provide us with adjunctions between $1$-reduced simplicial sets and (pointed) simplicial sets that we rely on to bridge our investigation of the homotopy theory of $1$-reduced simplicial sets to known results from the homotopy theory of (pointed) simplicial sets.

\subsection{Simplicial sets}\label{section2.1} We collect some background and notation on simplicial sets; a reference is \cite[Chapter I]{GoerssJardine1999}. The combinatorics involved in the definition of a simplicial set can be packaged in the simplex category $\Delta$.

\begin{definition}\label{defsimpcatdef} The \emph{simplex category} $\Delta$ is the category whose objects are the sets $[n]=\{0, 1, \dots, n\}$ for a non-negative integer $n$ and whose morphisms are the order-preserving maps of these sets.
\end{definition}

The order-preserving maps of finite total orders can be built out of special injections and surjections.

\begin{example}\label{defsimpcatdef1} For $0 \leq i \leq n$, two special maps in $\Delta$ are the $i$-th \emph{coface map} $d^{i} \colon [n-1] \rightarrow [n]$, given by \[d^{i}(x)= \begin{cases} x, & x<i; \\
x+1, & x \geq i,
\end{cases}\]
and the $i$-th \emph{codegeneracy map} $s^{i} \colon [n+1] \rightarrow [n]$, given by \[s^{i}(x)= \begin{cases} x, & x \leq i; \\
x-1, & x>i.
\end{cases}\]
These maps generate $\Delta$ subject to a list of identities on their composites.
\end{example}
The simplex category $\Delta$ yields a concise definition of a simplicial set; we write ${\mathcal{S}\mathrm{ets}}$ for the category of sets.

\begin{definition}\label{defssetdef} A \emph{simplicial set} is a functor $X \colon \Delta^{\operatorname{op}} \rightarrow {\mathcal{S}\mathrm{ets}}$. In detail, a simplicial set comprises a \emph{set of $n$-simplices} $X_{n}$ for every non-negative integer $n$, \emph{face maps} $d_{i} \colon X_{n} \rightarrow X_{n-1}$ for $0 \leq i \leq n$, and \emph{degeneracy maps} $s_{i} \colon X_{n} \rightarrow X_{n+1}$ for $0 \leq i \leq n$ subject to a list of identities on how the faces and degeneracies are composed.
\end{definition}
We write ${\mathcal{SS}\mathrm{ets}}$ for the category of simplicial sets and their simplicial maps (natural transformations). We give the fundamental building block examples of simplicial sets, which are reminiscent of simplicial complexes.

\begin{example}\label{defssetdef1} The \emph{standard $n$-simplex} is $\Delta[n]=\operatorname{Hom}_{\Delta}(-, [n]);$ the identity $\operatorname{id}_{[n]}$ gives an $n$-simplex $\sigma_{n}$. The \emph{boundary} $\partial\Delta[n]$ is the smallest subcomplex of $\Delta[n]$ containing $d_{i}(\sigma_{n})$ for $0 \leq i \leq n$. For $n \geq 1$ and $0 \leq k \leq n$, the \emph{horn} $\Lambda[n, k]$ is the smallest subcomplex of $\Delta[n]$ containing $d_{i}(\sigma_{n})$ for $i \neq k$.
\end{example}
As is the case for simplicial complexes, we can geometrically realize a simplicial set $X$ to obtain a CW complex $|X|$. Writing ${\mathcal{T}\mathrm{op}}$ for the category of topological spaces, we denote the geometric realization functor by $|-| \colon {\mathcal{SS}\mathrm{ets}} \rightarrow {\mathcal{T}\mathrm{op}}$. Its right adjoint sends a topological space $Y$ to the {singular complex} $\operatorname{Sing}(Y)$ with $n$-simplices $\operatorname{Sing}(Y)_{n}=\operatorname{Hom}_{\mathcal{T}\mathrm{op}}(|\Delta[n]|, Y)$. As the name suggests, there is a connection between the singular complex of a space and its singular homology; the $n$-th abelian group in the singular chain complex of a space is defined to be the free abelian group on the set of $n$-simplices of the singular complex of that space.

\subsection{The $1$-skeleton and the $1$-coskeleton of a simplicial set}\label{section2.2} To define $1$-reduced simplicial sets, we recall the $1$-skeleton of a simplicial set. We also recall the $1$-coskeleton of a simplicial set, which we refer to in the definition of the $1$st Eilenberg subcomplex of a pointed simplicial set. Our reference is \cite[Chapter 2.6]{Bergner2018}.

Let $\Delta_{1}$ denote the full subcategory of $\Delta$ with objects $[0]$ and $[1]$. The inclusion $\Delta_{1} \rightarrow \Delta$ induces a truncation functor $\operatorname{tr}_{1} \colon {\mathcal{SS}\mathrm{ets}} \rightarrow {\mathcal{S}\mathrm{ets}}^{\Delta_{1}^{\operatorname{op}}}$ that admits both a left adjoint $s_{1} \colon {\mathcal{S}\mathrm{ets}}^{\Delta_{1}^{\operatorname{op}}} \rightarrow {\mathcal{SS}\mathrm{ets}}$ and a right adjoint $c_{1} \colon {\mathcal{S}\mathrm{ets}}^{\Delta_{1}^{\operatorname{op}}} \rightarrow {\mathcal{SS}\mathrm{ets}}$, in terms of which we can define the $1$-skeleton and the $1$-coskeleton of a simplicial set.

\begin{definition}\label{def1skeletondef} Let $X$ be a simplicial set. The \emph{$1$-skeleton} of $X$ is $\operatorname{sk}_{1}(X)=s_{1}\operatorname{tr}_{1}(X)$. The \emph{$1$-coskeleton} of $X$ is $\operatorname{cosk}_{1}(X)=c_{1}\operatorname{tr}_{1}(X)$.
\end{definition}
In detail, both $\operatorname{sk}_{1}(X)$ and $\operatorname{cosk}_{1}(X)$ have set of $0$-simplices $X_{0}$ and set of $1$-simplices $X_{1}$, but they differ in their higher-dimensional simplices. The higher-dimensional simplices of $\operatorname{sk}_{1}(X)$ are the iterated degeneracies of its $1$-simplices. On the contrary, $\operatorname{cosk}_{1}(X)$ has an $n$-simplex for every compatible collection of $(n-1)$-simplices. More precisely, $\operatorname{cosk}_{1}(X)$ has an $n$-simplex for every map $\operatorname{sk}_{1}(\Delta[n]) \rightarrow X$.

\begin{remark}\label{def1skeletondef0} It follows that, for a simplicial set $X$, we have an inclusion map $\operatorname{sk}_{1}(X) \rightarrow X$ and a canonical map $X \rightarrow \operatorname{cosk}_{1}(X)$, where the latter map identifies the simplices of $X$ that share the same $1$-skeleton.
\end{remark}

Besides the above remark, we also need the following example for later use in this section.

\begin{example}\label{def1skeletondef1} Observe that the $1$-skeleton of the standard $2$-simplex $\Delta[2]$ is its boundary $\partial \Delta[2]$.
\end{example}

\subsection{Definition of $1$-reduced simplicial sets}\label{section2.3} We are now ready to introduce $1$-reduced simplicial sets.

\begin{definition}\label{defreduceddef} A simplicial set $X$ is \emph{$1$-reduced} if its $1$-skeleton $\operatorname{sk}_{1}(X)$ is a point. In detail, a simplicial set $X$ is $1$-reduced if it has a unique $0$-simplex whose degeneracy is the unique $1$-simplex of $X$.
\end{definition}

\begin{remark}\label{defreducedrefrmk} Note that the geometric realization of a $1$-reduced simplicial set is a $1$-connected CW complex. It is also crucial to observe that a $1$-reduced simplicial set is unambiguously pointed, for it has a unique $0$-simplex.
\end{remark}
We write ${\mathcal{SS}\mathrm{ets}}_{1}$ for the category of $1$-reduced simplicial sets and simplicial maps between them. The most important example of a $1$-reduced simplicial set for our work is the following.

\begin{example}\label{sothisdoeswork} For $n \geq 2$, the \emph{simplicial $n$-sphere} $S_{\mathrm{simp}}^{n}=\Delta[n]/\partial\Delta[n]$ is a $1$-reduced simplicial set consisting of a unique $0$-simplex and a unique non-degenerate $n$-simplex.
\end{example}
We also need the following lemma, whose proof is straightforward; we write $\wedge$ to denote the smash product.

\begin{lemma}\label{smashingtensor} If $X$ is a $1$-reduced simplicial set and $Z$ is a pointed simplicial set, then $X \wedge Z$ is $1$-reduced.
\end{lemma}

\subsection{The $1$-reduction of a simplicial set}\label{section2.4}
We relate the category ${\mathcal{SS}\mathrm{ets}}_{1}$ of $1$-reduced simplicial sets to the category ${\mathcal{SS}\mathrm{ets}}$ of simplicial sets by means of a left adjoint to the forgetful functor $U \colon {\mathcal{SS}\mathrm{ets}}_{1} \rightarrow {\mathcal{SS}\mathrm{ets}}$. This left adjoint quotients out the $1$-skeleton of a simplicial set to yield a $1$-reduced simplicial set. Although quotients can be homotopically ill-behaved, this adjunction aids our model-categorical arguments in this paper, for it lets us translate lifting problems that we wish to solve in ${\mathcal{SS}\mathrm{ets}}_{1}$ into lifting problems that we can solve in ${\mathcal{SS}\mathrm{ets}}$.

\begin{definition}\label{defreductiondef} The \emph{$1$-reduction} of a simplicial set $X$ is the $1$-reduced simplicial set $R_{1}(X)=X/\operatorname{sk}_{1}(X)$.
\end{definition}

The following observation carries useful consequences for later, so we formulate it explicitly.

\begin{example}\label{defreductiondef2} The simplicial $2$-sphere is $S_{\mathrm{simp}}^{2}=R_{1}(\Delta[2])$ because, by Example \ref{def1skeletondef1}, $\operatorname{sk}_{1}(\Delta[2])=\partial\Delta[2]$.
\end{example}
The $1$-reduction of a simplicial set is its largest $1$-reduced quotient in the sense of the following universal property, which is a consequence of the universal property of quotients.

\begin{proposition}\label{unproprreduction} For every map $f \colon X \rightarrow Y$ from a simplicial set $X$ to a $1$-reduced simplicial set $Y$, there exists a unique map $\overline{f} \colon R_{1}(X) \rightarrow Y$ such that the diagram
\[\begin{tikzcd}[row sep=3mm, column sep=3mm, ampersand replacement=\&]
	X \&\& {R_{1}(X)} \\
	\\
	\&\& {Y}
	\arrow[from=1-1, to=1-3]
	\arrow["f"', from=1-1, to=3-3]
	\arrow["{\overline{f}}"', dashed, from=1-3, to=3-3]
\end{tikzcd}\]
commutes, where $X \rightarrow R_{1}(X)$ is the quotient map from $X$ to $R_{1}(X)=X/\operatorname{sk}_{1}(X)$.
\end{proposition}
This universal property translates to the promised adjunction.

\begin{proposition}\label{notquite} The $1$-reduction ${R}_{1} \colon {\mathcal{SS}\mathrm{ets}} \rightarrow {\mathcal{SS}\mathrm{ets}}_{1}$ is left adjoint to the forgetful functor.
\end{proposition}
We invoke Example \ref{defreductiondef2} to note a consequence of this adjunction for later use.

\begin{corollary}\label{notquitesimplices} If $X$ is a $1$-reduced simplicial set, then a map $\Delta[2] \rightarrow X$ is adjoint to a map $S_{\mathrm{simp}}^{2} \rightarrow X$.
\end{corollary}

\begin{proof} The claim is a consequence of Proposition \ref{notquite} because $S_{\mathrm{simp}}^{2}=R_{1}(\Delta[2])$ by Example \ref{defreductiondef2}.
\end{proof}

\subsection{The $1$st Eilenberg subcomplex of a pointed simplicial set}\label{section2.5}
As we shall see, the approach of taking the $1$-reduction of a simplicial set does not play as nicely as we need it to for our homotopy-theoretic purposes, but we can do better. Because a $1$-reduced simplicial set has a unique $0$-simplex, there is a forgetful functor $F \colon {\mathcal{SS}\mathrm{ets}}_{1} \rightarrow {\mathcal{SS}\mathrm{ets}}_{\ast}$ from the category ${\mathcal{SS}\mathrm{ets}}_{1}$ of $1$-reduced simplicial sets to the category ${\mathcal{SS}\mathrm{ets}}_{\ast}$ of pointed simplicial sets. This functor also has a right adjoint given by the $1$st Eilenberg subcomplex of a pointed simplicial set, which is its largest $1$-reduced subcomplex containing the basepoint.

Note that, in the absence of a basepoint, there is generally no canonical choice of such a subcomplex, so the use of pointed simplicial sets is essential for this construction. In fact, we shall see that the forgetful functor from $1$-reduced simplicial sets to simplicial sets has no right adjoint.

We present the definition of the $1$st Eilenberg subcomplex, though the only feature that we need in this paper is the adjunction that it features in. Recall from Remark \ref{def1skeletondef0} that, for a simplicial set $X$, we have a canonical map $X \rightarrow \operatorname{cosk}_{1}(X)$ to its $1$-coskeleton.

\begin{definition}[{\cite[Definition 8.3]{May1967}}]\label{defeilenbergsubdef} The \emph{$1$st Eilenberg subcomplex} $E_{1}(X)$ of a pointed simplicial set $X$ with basepoint $x_{0}$ is defined by the pullback

\[\begin{tikzcd}[row sep=3mm, column sep=3mm, ampersand replacement=\&]
	{E_{1}(X)} \&\& X \\
	\\
	{\Delta[0]} \&\& {\operatorname{cosk}_{1}(X).}
	\arrow[from=1-1, to=1-3]
	\arrow[from=1-1, to=3-1]
	\arrow[from=1-3, to=3-3]
	\arrow["{x_{0}}", from=3-1, to=3-3]
\end{tikzcd}\]
In detail, $E_{1}(X)$ consists of all simplices $\chi$ of $X$ such that every $0$-dimensional face of $\chi$ is $x_{0}$ and every $1$-dimensional face of $\chi$ is the degeneracy of $x_{0}$.
\end{definition}

It follows that the $1$st Eilenberg subcomplex $E_{1}(X)$ of a pointed simplicial set $X$ is $1$-reduced. The $1$st Eilenberg subcomplex of a pointed simplicial set is its largest $1$-reduced subcomplex in the sense of the following universal property that is straightforward to verify.

\begin{proposition}\label{unproprEilenberg} Every pointed map $f \colon X \rightarrow Y$ from a $1$-reduced simplicial set $X$ to a pointed simplicial set $Y$ maps into the $1$st Eilenberg subcomplex $E_{1}(Y)$ of $Y$. Thus, there is a unique map $\hat{f} \colon X \rightarrow E_{1}(Y)$ such that the diagram
\[\begin{tikzcd}[row sep=3mm, column sep=3mm, ampersand replacement=\&]
	X \&\& {E_{1}(Y)} \\
	\\
	\&\& {Y}
	\arrow["{\hat{f}}", dashed, from=1-1, to=1-3]
	\arrow["f"', from=1-1, to=3-3]
	\arrow[from=1-3, to=3-3]
\end{tikzcd}\]
commutes, where $E_{1}(Y) \rightarrow Y$ is the inclusion map.
\end{proposition}

This universal property amounts to the promised adjunction, which is in part (2) of the following lemma.

\begin{proposition}\label{yesquite} Let $F \colon {\mathcal{SS}\mathrm{ets}}_{1} \rightarrow {\mathcal{SS}\mathrm{ets}}_{\ast}$ denote the forgetful functor.

\begin{enumerate}

\item The $1$-reduction functor ${R}_{1} \colon {\mathcal{SS}\mathrm{ets}_{\ast}} \rightarrow {\mathcal{SS}\mathrm{ets}}_{1}$ is left adjoint to $F$.

\item The $1$st Eilenberg subcomplex functor $E_{1} \colon {\mathcal{SS}\mathrm{ets}_{\ast}} \rightarrow {\mathcal{SS}\mathrm{ets}}_{1}$ is right adjoint to $F$.
\end{enumerate}

\end{proposition}
Proposition \ref{yesquite} implies that small limits and colimits in ${\mathcal{SS}\mathrm{ets}}_{1}$ are constructed as in ${\mathcal{SS}\mathrm{ets}}_{\ast}$. Given a diagram $L \colon \mathcal{D} \rightarrow {\mathcal{SS}\mathrm{ets}}_{1}$ in $1$-reduced simplicial sets, we form its (co)limit by post-composing $L$ with the forgetful functor $F \colon {\mathcal{SS}\mathrm{ets}}_{1} \rightarrow {\mathcal{SS}\mathrm{ets}}_{\ast}$ and forming the (co)limit of the diagram $FL \colon \mathcal{D} \rightarrow {\mathcal{SS}\mathrm{ets}}_{\ast}$ in pointed simplicial sets. In particular, coproducts in ${\mathcal{SS}\mathrm{ets}}_{1}$ are given by wedge sums instead of disjoint unions, for the latter do not yield a $1$-reduced simplicial set. More fundamentally, the initial object of ${\mathcal{SS}\mathrm{ets}}_{1}$ is the point $\Delta[0]$; the empty set $\emptyset$ is not $1$-reduced. Thus, the forgetful functor $U \colon {\mathcal{SS}\mathrm{ets}}_{1} \rightarrow {\mathcal{SS}\mathrm{ets}}$ does not preserve the initial object or coproducts, so it cannot admit a right adjoint.

\section{The left transfer to $1$-reduced simplicial sets}\label{section3}
We recover the first model structure on $1$-reduced simplicial sets from Quillen's family, which we then localize to recover the rest of the family. Namely, we left-transfer the model structure on the category ${\mathcal{SS}\mathrm{ets}}_{\ast}$ of pointed simplicial sets along the adjunction \[F \colon {\mathcal{SS}\mathrm{ets}}_{1} \rightleftarrows {\mathcal{SS}\mathrm{ets}}_{\ast} \colon E_{1}\] to a model structure on the category ${\mathcal{SS}\mathrm{ets}}_{1}$ of $1$-reduced simplicial sets whose weak equivalences are the weak homotopy equivalences. (Our notation for adjunctions is amenable to a helpful mnemonic that distinguishes a left transfer from a right transfer; we propagate the existence of a model structure from right \emph{to left}, so it is a \emph{left} transfer.) This left transfer is an application of \cite[Theorem 2.23]{BHKKRS2015}, which we recall after going over the background on model category theory that we rely on in this paper.

\subsection{Model categories}\label{section3.1} We collect some background on model categories from \cite[\S 3]{DwyerSpalinski1995} and \cite[Chapters 7, 11, and 13]{Hirschhorn2003}. We first recall a concept in the homotopy theory of simplicial sets from \cite[Chapter I]{GoerssJardine1999}.
\begin{definition}\label{defkanfibrdef} A \emph{Kan fibration} is a simplicial map $p \colon X \rightarrow Y$ that has the right lifting property with respect to the set of all horn inclusions
$J=\{j_{n, k} \colon \Lambda[n, k] \rightarrow \Delta[n] \mid n \geq 1 \textrm{ and } 0 \leq k \leq n\}$. In other words, a simplicial map $p \colon X \rightarrow Y$ is a Kan fibration if every lifting problem of the form
\[\begin{tikzcd}[row sep=3mm, column sep=3mm, ampersand replacement=\&]
	{\Lambda[n, k]} \&\& X \\
	\\
	{\Delta[n]} \&\& Y
	\arrow[from=1-1, to=1-3]
	\arrow["{j_{n, k}}"', from=1-1, to=3-1]
	\arrow["p"', from=1-3, to=3-3]
	\arrow[dashed, from=3-1, to=1-3]
	\arrow[from=3-1, to=3-3]
\end{tikzcd}\]
can be solved. A \emph{Kan complex} is a simplicial set $X$ such that the unique map $X \rightarrow \Delta[0]$ is a Kan fibration.
\end{definition}

\begin{example}[{\cite[Chapter I, Lemma 3.3]{GoerssJardine1999}}]\label{defkanfibrdef2} If $Y$ is a space, then $\operatorname{Sing}(Y)$ is a Kan complex.
\end{example}

Example \ref{defkanfibrdef2} follows from the geometric realization-singular complex adjunction combined with the fact that, at the level of topological spaces, geometric realizations of horns can always be filled because they are retracts of geometric realizations of standard simplices. However, the directionality of simplices in a simplicial set precludes the existence of the analogous retraction at the level of simplicial sets, so the concepts of a Kan fibration and a Kan complex are not moot.

The homotopy theory of simplicial sets can be studied systematically in the framework of model categories, which is the framework that we adhere to throughout this paper.

\begin{definition}\label{defmodelcatdef} A \emph{model category} comprises a category $\mathcal{M}$ that has all finite limits and colimits, as well as three distinguished classes of morphisms in $\mathcal{M}$: a class of \emph{weak equivalences}, a class of \emph{cofibrations}, and a class of \emph{fibrations}. A weak equivalence that is also a cofibration is called an \emph{acyclic cofibration}. A weak equivalence that is also a fibration is called an \emph{acyclic fibration}. This data satisfies the axioms in \cite[Definition 3.3]{DwyerSpalinski1995}.
\end{definition}

\begin{remark}\label{defmodelcatdefrmk2} In a model category, the weak equivalences and the cofibrations together determine the fibrations.
\end{remark}

To recover Quillen's family of model structures on $1$-reduced simplicial sets, we rely on the standard model structure on (pointed) simplicial sets.

\begin{example}[{\cite[Theorems 7.10.12 and 7.10.13]{Hirschhorn2003}}]\label{defmodelcatdef1} The category ${\mathcal{SS}\mathrm{ets}}$ of simplicial sets has a model structure in which a map $f$ is

\begin{enumerate}

\item a weak equivalence if and only if $|f|$ is a weak homotopy equivalence;

\item a cofibration if and only if $f$ is a monomorphism; and

\item a fibration if and only if $f$ is a Kan fibration.

\end{enumerate}
The category ${\mathcal{SS}\mathrm{ets}}_{\ast}$ of pointed simplicial sets also has a model structure with the same description.

\end{example}

The following definition is intended to introduce some terminology for the rest of this paper.

\begin{definition}\label{defmodelcatdefrmk1} Let $\mathcal{M}$ be a model category with initial object $\emptyset$ and terminal object $\ast$. An object $X$ of $\mathcal{M}$ is \emph{cofibrant} if the morphism $\emptyset \rightarrow X$ is a cofibration, and $X$ is \emph{fibrant} if the morphism $X \rightarrow \ast$ is a fibration. Then, for an object $Y$ of $\mathcal{M}$, a \emph{cofibrant replacement} of $Y$ is a weak equivalence $A \rightarrow Y$ with cofibrant source $A$, and a \emph{fibrant replacement} of $Y$ is a weak equivalence $Y \rightarrow B$ with fibrant target $B$.
\end{definition}
The model categories ${\mathcal{SS}\mathrm{ets}}$ and ${\mathcal{SS}\mathrm{ets}}_{\ast}$ have the property of being cofibrantly generated. In practice, this property facilitates the detection of (acyclic) fibrations and fibrant objects in a model category.

\begin{definition}[{\cite[Definition 11.1.2]{Hirschhorn2003}}]\label{defcofgenmcdef} A model category $\mathcal{M}$ is \emph{cofibrantly generated} if there exist

\begin{enumerate}

\item a set $I$ of morphisms in $\mathcal{M}$, called a \emph{generating set of cofibrations} for $\mathcal{M}$; and

\item a set $J$ of morphisms in $\mathcal{M}$, called a \emph{generating set of acyclic cofibrations} for $\mathcal{M}$,
\end{enumerate}
such that

\begin{enumerate}

\item both $I$ and $J$ permit the small object argument \cite[Definition 10.5.15]{Hirschhorn2003};

\item a morphism $p$ is an acyclic fibration if and only if $p$ has the right lifting property with respect to $I$; and

\item a morphism $p$ is a fibration if and only if $p$ has the right lifting property with respect to $J$.
\end{enumerate}

\end{definition}

\begin{example}[{\cite[Examples 11.1.6 and 11.1.7]{Hirschhorn2003}}]\label{defcofgenmcdef1} Both ${\mathcal{SS}\mathrm{ets}}$ and ${\mathcal{SS}\mathrm{ets}}_{\ast}$ are cofibrantly generated.

\begin{enumerate}

\item A generating set of cofibrations for ${\mathcal{SS}\mathrm{ets}}$ is the set of all boundary inclusions \[I=\{i_{n} \colon \partial\Delta[n] \rightarrow \Delta[n] \mid n \geq 0\}.\]

\item A generating set of acyclic cofibrations for ${\mathcal{SS}\mathrm{ets}}$ is the set of all horn inclusions \[J=\{j_{n, k} \colon \Lambda[n, k] \rightarrow \Delta[n] \mid n \geq 1 \textrm{ and } 0 \leq k \leq n\}.\]

\end{enumerate}
Attaching a disjoint basepoint to the maps in both sets yields the generating (acyclic) cofibrations of ${\mathcal{SS}\mathrm{ets}}_{\ast}$.

\end{example}
What is more, ${\mathcal{SS}\mathrm{ets}}$ and ${\mathcal{SS}\mathrm{ets}}_{\ast}$ belong in the following class of well-behaved model categories, which is originally due to J.H. Smith and is particularly expedient for our purposes of left Bousfield localization.

\begin{definition}[{\cite[Definition 2.1]{Dugger2001}}]\label{defcombmodelcatdef} A model category $\mathcal{M}$ is \emph{combinatorial} if $\mathcal{M}$ is cofibrantly generated and its underlying category is locally presentable \cite[Definition 1.17]{AdamekRosicky1994}.
\end{definition}

In a combinatorial model category, not only do we have a set of building blocks for its (acyclic) cofibrations by virtue of cofibrant generation, but we also have a set of building blocks for its objects by virtue of local presentability. Though the category of topological spaces is not locally presentable \cite[Example 1.18(5)]{AdamekRosicky1994}, the category of simplicial sets does not run into this problem as every simplicial set is built out of standard simplices.

\begin{example}[{\cite[Example 1.12]{AdamekRosicky1994}}]\label{defcombmodelcatdef1} Both ${\mathcal{SS}\mathrm{ets}}$ and ${\mathcal{SS}\mathrm{ets}}_{\ast}$ are locally presentable, thus combinatorial.
\end{example}
Also, ${\mathcal{SS}\mathrm{ets}}$ and ${\mathcal{SS}\mathrm{ets}}_{\ast}$ have the following model-categorical property that permits left Bousfield localization.

\begin{definition}[{\cite[Definition 13.1.1(1)]{Hirschhorn2003}}]\label{deflpropmodelcatdef} A model category $\mathcal{M}$ is \emph{left proper} if every pushout of a weak equivalence along a cofibration is a weak equivalence.
\end{definition}
The following result guarantees the left properness of almost every model category that appears in this paper, including every model structure in Quillen's family of model structures on $1$-reduced simplicial sets.

\begin{proposition}[{\cite[Corollary 13.1.3(1)]{Hirschhorn2003}}]\label{leftpropernoworries} If every object of a model category $\mathcal{M}$ is cofibrant, then $\mathcal{M}$ is left proper.
\end{proposition}

In practice, it is common and convenient for the cofibrations in a model category to be the monomorphisms, as is the theme in this paper. Then, every object is cofibrant, so we need not worry about cofibrant replacement, and we can use Proposition \ref{leftpropernoworries} to extract left properness for free.

\subsection{The left transfer}\label{section3.2}
Before we do the promised left transfer, we explain why we cannot use the $1$-reduction functor $R_{1}$ to right-induce, in the sense of \cite[Theorem 11.3.2]{Hirschhorn2003}, a model structure on ${\mathcal{SS}\mathrm{ets}}_{1}$ along the adjunction \[{R}_{1} \colon {\mathcal{SS}\mathrm{ets}} \rightleftarrows {\mathcal{SS}\mathrm{ets}}_{1} \colon U\] from Proposition \ref{notquite}. We write $J$ for the set of all horn inclusions in ${\mathcal{SS}\mathrm{ets}}$. If such a right transfer were possible, then, in the resulting right-induced model structure on ${\mathcal{SS}\mathrm{ets}}_{1}$,

\begin{enumerate}

\item the class of weak equivalences would be the class of weak homotopy equivalences; and

\item the set $R_{1}(J)$ would be a generating set of acyclic cofibrations.

\end{enumerate}
In particular, the set $|R_{1}(J)|$ would comprise weak homotopy equivalences, which is false by part (2) of the following lemma; conceptually, we find an instance of the $1$-reduction quotient being homotopically ill-behaved.

\begin{lemma}\label{sadcarefree} For $n \geq 1$ and $0 \leq k \leq n$, let $j_{n, k} \colon \Lambda[n, k] \rightarrow \Delta[n]$ denote the horn inclusion in ${\mathcal{SS}\mathrm{ets}}$.

\begin{enumerate}

\item If $n=1$, then ${R}_{1}(j_{1, k})$ is the identity map of the point.

\item If $n=2$, then ${R}_{1}(j_{2, k}) \colon \Delta[0] \rightarrow S_{\mathrm{simp}}^{2}$ picks out the unique $0$-simplex of the simplicial $2$-sphere.

\item If $n \geq 3$, then $|{R}_{1}(j_{n, k})|$ is the inclusion of a strong deformation retract in topological spaces.

\end{enumerate}
Thus, $|{R}_{1}(j_{n, k})|$ is a weak homotopy equivalence if and only if $n \neq 2$.

\end{lemma}

\begin{proof} Parts (1) and (2) are straightforward. For (3), if $n \geq 3$, then $\operatorname{sk}_{1}(\Lambda[n, k])=\operatorname{sk}_{1}(\Delta[n])$, so the inclusion of a strong deformation retract $|j_{n, k}|$ descends to the inclusion of a strong deformation retract $|{R}_{1}(j_{n, k})|$.
\end{proof}
On that note, we also record a computation of a similar nature for future reference in this section.

\begin{lemma}\label{notallislost} If a simplicial map $f \colon X \rightarrow Y$ is a monomorphism, then so is its $1$-reduction ${R}_{1}(f)$.
\end{lemma}
Instead, we left-induce a model structure on ${\mathcal{SS}\mathrm{ets}}_{1}$ by transferring that on ${\mathcal{SS}\mathrm{ets}}_{\ast}$ along the adjunction \[F \colon {\mathcal{SS}\mathrm{ets}}_{1} \rightleftarrows {\mathcal{SS}\mathrm{ets}}_{\ast} \colon E_{1}\] from Proposition \ref{yesquite}(2). To do so, we apply the following general result on left-induced model structures.

\begin{theorem}[{\cite[Theorem 2.23]{BHKKRS2015}}]\label{leftinducetheoremwit} Let $F \colon \mathcal{C} \rightleftarrows \mathcal{M} \colon G$ be an adjunction, where $\mathcal{C}$ is a locally presentable category and $\mathcal{M}$ is a combinatorial model category with class of cofibrations $Q$. Suppose that, if a map $g$ in $\mathcal{C}$ has the right lifting property with respect to the class $F^{-1}(Q)$, then $F(g)$ is a weak equivalence in $\mathcal{M}$. Then, there is a combinatorial model structure on $\mathcal{C}$, called the \emph{left-induced model structure}, in which a map $f$ is

\begin{enumerate}

\item a weak equivalence if and only if $F(f)$ is a weak equivalence in $\mathcal{M}$; and

\item a cofibration if and only if $F(f)$ is a cofibration in $\mathcal{M}$.

\end{enumerate}
\end{theorem}
As ${\mathcal{SS}\mathrm{ets}}_{1}$ is a locally presentable category \cite[Example 1.12]{AdamekRosicky1994} and ${\mathcal{SS}\mathrm{ets}}_{\ast}$ is a combinatorial model category by Example \ref{defcombmodelcatdef1}, it remains to verify that our adjunction satisfies the homotopical condition in Theorem \ref{leftinducetheoremwit}.

\begin{lemma}\label{verifybhkkrs} Let $f \colon X \rightarrow Y$ be a map of $1$-reduced simplicial sets. If $f$ has the right lifting property with respect to every monomorphism of $1$-reduced simplicial sets, then $|f|$ is a weak homotopy equivalence.
\end{lemma}

\begin{proof} We show that $f$ is an acyclic fibration in ${\mathcal{SS}\mathrm{ets}}$. For $n \geq 0$, let $i_{n} \colon \partial\Delta[n] \rightarrow \Delta[n]$ denote the boundary inclusion in ${\mathcal{SS}\mathrm{ets}}$. By Lemma \ref{notallislost}, ${R}_{1}(i_{n})$ is a monomorphism of $1$-reduced simplicial sets. Therefore, the hypothesis informs us that, for every $n \geq 0$, $f$ has the right lifting property with respect to ${R}_{1}(i_{n})$. Then, the adjunction ${R}_{1} \colon {\mathcal{SS}\mathrm{ets}} \rightleftarrows {\mathcal{SS}\mathrm{ets}}_{1} \colon U$ from Proposition \ref{notquite} yields that, for every $n \geq 0$, $f$ has the right lifting property with respect to the boundary inclusion $i_{n}$. We conclude that $f$ is an acyclic fibration in ${\mathcal{SS}\mathrm{ets}}$.
\end{proof}
We begin recovering the family of model structures on $1$-reduced simplicial sets in \cite[Part II, Theorem 2.2]{Quillen1969}.

\begin{theorem}\label{rightinduceeyebrow} There is a combinatorial and left proper model structure on the category ${\mathcal{SS}\mathrm{ets}}_{1}$ of $1$-reduced simplicial sets in which a map $f$ is

\begin{enumerate}

\item a weak equivalence if and only if $|f|$ is a weak homotopy equivalence; and

\item a cofibration if and only if $f$ is a monomorphism.

\end{enumerate}
In this model structure, a $1$-reduced simplicial set $X$ is fibrant if and only if $X$ is a Kan complex.
\end{theorem}

\begin{proof} Theorem \ref{leftinducetheoremwit} yields every claim other than the characterization of fibrant objects. Every $1$-reduced Kan complex $X$ is fibrant in ${\mathcal{SS}\mathrm{ets}}_{1}$ because, by virtue of being a Kan complex, $X$ has the right lifting property with respect to the class of acyclic cofibrations in ${\mathcal{SS}\mathrm{ets}}_{1}$. Conversely, if a $1$-reduced simplicial set $X$ is fibrant in ${\mathcal{SS}\mathrm{ets}}_{1}$, then we show that $X$ is a Kan complex. For $n \geq 1$ and $0 \leq k \leq n$, we solve a lifting problem
\[\begin{tikzcd}[row sep=3mm, column sep=3mm, ampersand replacement=\&]
	{\Lambda[n, k]} \&\& X \\
	\\
	{\Delta[n].}
	\arrow[from=1-1, to=1-3]
	\arrow["{j_{n, k}}"', from=1-1, to=3-1]
	\arrow[dashed, from=3-1, to=1-3]
\end{tikzcd}\]
Using the adjunction ${R}_{1} \colon {\mathcal{SS}\mathrm{ets}} \rightleftarrows {\mathcal{SS}\mathrm{ets}}_{1} \colon U$ from Proposition \ref{notquite}, we equivalently solve a lifting problem
\[\begin{tikzcd}[row sep=3mm, column sep=3mm, ampersand replacement=\&]
	{{R}_{1}(\Lambda[n, k])} \&\& X \\
	\\
	{{R}_{1}(\Delta[n]).}
	\arrow[from=1-1, to=1-3]
	\arrow["{{R}_{1}(j_{n, k})}"', from=1-1, to=3-1]
	\arrow[dashed, from=3-1, to=1-3]
\end{tikzcd}\]
In light of Lemma \ref{sadcarefree}, we split our solution in two separate cases. If $n \neq 2$, then, by Lemmas \ref{sadcarefree} and \ref{notallislost}, ${R}_{1}(j_{n, k})$ is an acyclic cofibration in ${\mathcal{SS}\mathrm{ets}}_{1}$, so our lifting problem has a solution since $X$ is fibrant in ${\mathcal{SS}\mathrm{ets}}_{1}$. If $n=2$, then we solve a lifting problem
\[\begin{tikzcd}[row sep=3mm, column sep=3mm, ampersand replacement=\&]
	{\Delta[0]} \&\& X \\
	\\
	{S_{\mathrm{simp}}^{2},}
	\arrow["{x_{0}}", from=1-1, to=1-3]
	\arrow["{{R}_{1}(j_{2, k})}"', from=1-1, to=3-1]
	\arrow["{\psi}", dashed, from=3-1, to=1-3]
\end{tikzcd}\]
where $x_{0}$ denotes the unique $0$-simplex of $X$, by setting $\psi$ to be the constant map at $x_{0}$.
\end{proof}
The left transfer renders $F \colon {\mathcal{SS}\mathrm{ets}}_{1} \rightleftarrows {\mathcal{SS}\mathrm{ets}}_{\ast} \colon E_{1}$ a Quillen adjunction \cite[Definition 8.5.2]{Hirschhorn2003}, so $E_{1}$ preserves fibrant objects; an alternative proof of this preservation uses the pullback definition of the $1$st Eilenberg subcomplex together with the fact that, if $X$ is a Kan complex, then $X \rightarrow \operatorname{cosk}_{1}(X)$ is a Kan fibration. We record the following consequence for the next section.

\begin{corollary}\label{forlaterusetbacor} If $Y$ is a pointed topological space, then $E_{1}\operatorname{Sing}(Y)$ is a Kan complex.
\end{corollary}

\section{Fibrations of $1$-reduced simplicial sets}\label{section4}
We study the fibrations in the model category ${\mathcal{SS}\mathrm{ets}}_{1}$. Namely, we produce functorial fibrant replacements, a generating set of cofibrations, and a set of acyclic cofibrations that detects fibrant objects. Moreover, we characterize when a fibration in ${\mathcal{SS}\mathrm{ets}}_{1}$ is a Kan fibration, and we find a counterintuitive fibration in ${\mathcal{SS}\mathrm{ets}}_{1}$.

The existence of such counterintuitive fibrations in ${\mathcal{SS}\mathrm{ets}}_{1}$ serves not only as an impediment toward finding a generating set of acyclic cofibrations, but also as evidence that the model category ${\mathcal{SS}\mathrm{ets}}_{1}$ should not be thought of as a restriction of the model structure on pointed simplicial sets. Instead, ${\mathcal{SS}\mathrm{ets}}_{1}$ encodes its own homotopy theory that has its individual intricacies, some of which we shed light on in this section.

\subsection{Fibrant replacement of $1$-reduced simplicial sets}\label{section4.1} We invoke Corollary \ref{forlaterusetbacor} to produce functorial fibrant replacements in ${\mathcal{SS}\mathrm{ets}}_{1}$. Let $X$ be a $1$-reduced simplicial set. The geometric realization-singular complex adjunction provides the map $X \rightarrow \operatorname{Sing}(|X|)$, which lands into $E_{1}\operatorname{Sing}(|X|)$. We get the commutative diagram
\[\begin{tikzcd}[row sep=3mm, column sep=3mm, ampersand replacement=\&]
	X \&\& {E_{1}\operatorname{Sing}(|X|)} \\
	\\
	\&\& {\operatorname{Sing}(|X|),}
	\arrow[from=1-1, to=1-3]
	\arrow[from=1-1, to=3-3]
	\arrow[from=1-3, to=3-3]
\end{tikzcd}\]
where $E_{1}\operatorname{Sing}(|X|) \rightarrow \operatorname{Sing}(|X|)$ is the inclusion map. We show that the natural map \[X \rightarrow E_{1}\operatorname{Sing}(|X|)\] is a weak equivalence, thus a functorial fibrant replacement of $X$ in ${\mathcal{SS}\mathrm{ets}}_{1}$. To do so, it suffices to show that the other two maps in the commutative diagram are weak equivalences. Firstly, part (1) of the following theorem says that $X \rightarrow \operatorname{Sing}(|X|)$ is a weak equivalence; part (2) is for later use in our argument.

\begin{theorem}[{\cite[Chapter I, Theorem 11.4]{GoerssJardine1999}}]\label{quilleneqsecretly}

\begin{enumerate}

\item For every simplicial set $X$, the map $X \rightarrow \operatorname{Sing}(|X|)$ is a weak equivalence.

\item For every topological space $Y$, the map $|\operatorname{Sing}(Y)| \rightarrow Y$ is a weak homotopy equivalence.

\end{enumerate}

\end{theorem}

\begin{remark}\label{quilleneqsecretlyrmk} By Example \ref{defkanfibrdef2}, part (1) defines a functorial fibrant replacement in $\mathcal{SS}\mathrm{ets}$.
\end{remark}
The inclusion $E_{1}\operatorname{Sing}(|X|) \rightarrow \operatorname{Sing}(|X|)$ is also a weak equivalence by the following theorem.

\begin{theorem}[{\cite[Theorem 8.4(iii)]{May1967}}]\label{mayof67} If $X$ is a pointed Kan complex, then the $1$st Eilenberg subcomplex inclusion $E_{1}(X) \rightarrow X$ induces, for every $n \geq 2$, an isomorphism $\pi_{n}(E_{1}(X)) \xrightarrow{\cong} \pi_{n}(X)$. In particular, if $X$ is a pointed and $1$-connected Kan complex, then the inclusion $E_{1}(X) \rightarrow X$ is a weak equivalence.
\end{theorem}
In other words, if $X$ is a pointed Kan complex, then the inclusion $E_{1}(X) \rightarrow X$ is a $1$-connected cover of $X$. As an aside, using the construction of the homotopy groups of Kan complexes from \cite[Chapter I.7]{GoerssJardine1999} together with Proposition \ref{unproprEilenberg}, we give an alternative proof of Theorem \ref{mayof67} to that in {\cite[Theorem 8.4(iii)]{May1967}}.

\begin{proof} We first show that $\pi_{n}(E_{1}(X)) \rightarrow \pi_{n}(X)$ is an epimorphism. As $X$ is a Kan complex, every element of $\pi_{n}(X)$ is of the form $[\omega]$ for a map $\omega \colon S_{\mathrm{simp}}^{n} \rightarrow X$, and the universal property of the $1$st Eilenberg subcomplex yields the commutative diagram

\[\begin{tikzcd}[row sep=3mm, column sep=3mm, ampersand replacement=\&]
	{S_{\mathrm{simp}}^{n}} \&\& {E_{1}(X)} \\
	\\
	\&\& {X.}
	\arrow["{\hat{\omega}}", dashed, from=1-1, to=1-3]
	\arrow["{\omega}"', from=1-1, to=3-3]
	\arrow[from=1-3, to=3-3]
\end{tikzcd}\]
It remains to show that $\pi_{n}(E_{1}(X)) \rightarrow \pi_{n}(X)$ is a monomorphism. Given a homotopy $H \colon S_{\mathrm{simp}}^{n} \wedge \Delta[1]_{+} \rightarrow X$, we observe that $S_{\mathrm{simp}}^{n} \wedge \Delta[1]_{+}$ is $1$-reduced by Lemma \ref{smashingtensor}, so we obtain the commutative diagram

\[\begin{tikzcd}[row sep=3mm, column sep=3mm, ampersand replacement=\&]
	{S_{\mathrm{simp}}^{n} \wedge \Delta[1]_{+}} \&\& {E_{1}(X)} \\
	\\
	\&\& {X.}
	\arrow["{\hat{H}}", dashed, from=1-1, to=1-3]
	\arrow["{H}"', from=1-1, to=3-3]
	\arrow[from=1-3, to=3-3]
\end{tikzcd}\]
\end{proof}

We infer that the inclusion $E_{1}\operatorname{Sing}(|X|) \rightarrow \operatorname{Sing}(|X|)$ is also a weak equivalence.

\begin{corollary}\label{mayof67cor} If $Y$ is a pointed and $1$-connected space, then  $E_{1}\operatorname{Sing}(Y) \rightarrow \operatorname{Sing}(Y)$ is a weak equivalence.
\end{corollary}

\begin{proof} By Example \ref{defkanfibrdef2}, $\operatorname{Sing}(Y)$ is a Kan complex. By Theorem \ref{quilleneqsecretly}(2), $\operatorname{Sing}(Y)$ is $1$-connected.
\end{proof}
In conclusion, we obtain functorial fibrant replacements in ${\mathcal{SS}\mathrm{ets}}_{1}$.

\begin{proposition}\label{fibrepinogcat} Let $X$ be a $1$-reduced simplicial set. Then, the natural map $X \rightarrow E_{1}\operatorname{Sing}(|X|)$ is a functorial fibrant replacement of $X$ in ${\mathcal{SS}\mathrm{ets}}_{1}$.
\end{proposition}

\subsection{Generating cofibrations of $1$-reduced simplicial sets}\label{section4.2}
Our left transfer from Theorem \ref{rightinduceeyebrow} provides us with a combinatorial model structure on ${\mathcal{SS}\mathrm{ets}}_{1}$, but not with explicit generating sets of cofibrations and acyclic cofibrations. As a first remedy, we extract an explicit generating set of cofibrations for ${\mathcal{SS}\mathrm{ets}}_{1}$ from the set of boundary inclusions for ${\mathcal{SS}\mathrm{ets}}$ using the $1$-reduction functor $R_{1}$.

\begin{proposition}\label{gencofset} Let $I=\{i_{n} \colon \partial\Delta[n] \rightarrow \Delta[n] \mid n \geq 0\}$ denote the generating set of cofibrations for ${\mathcal{SS}\mathrm{ets}}$. Let $f \colon X \rightarrow Y$ be a map of $1$-reduced simplicial sets. The following statements are equivalent.

\begin{enumerate}

\item The map $f \colon X \rightarrow Y$ is an acyclic fibration in ${\mathcal{SS}\mathrm{ets}}$.

\item The map $f \colon X \rightarrow Y$ is an acyclic fibration in ${\mathcal{SS}\mathrm{ets}}_{1}$.

\item The map $f \colon X \rightarrow Y$ has the right lifting property with respect to the set \[{R}_{1}(I)=\{{R}_{1}(i_{n}) \colon {R}_{1}(\partial\Delta[n]) \rightarrow {R}_{1}(\Delta[n]) \mid n \geq 0\}.\]

\end{enumerate}
Hence, ${R}_{1}(I)$ is a generating set of cofibrations for ${\mathcal{SS}\mathrm{ets}}_{1}$.
\end{proposition}

\begin{proof} Firstly, we show that (1) implies (2). If $f$ is an acyclic fibration in ${\mathcal{SS}\mathrm{ets}}$, then $f$ has the right lifting property with respect to all monomorphisms in ${\mathcal{SS}\mathrm{ets}}$. In particular, $f$ has the right lifting property with respect to all monomorphisms in ${\mathcal{SS}\mathrm{ets}}_{1}$, so $f$ is an acyclic fibration in ${\mathcal{SS}\mathrm{ets}}_{1}$. Also, (2) implies (3) as, by Lemma \ref{notallislost}, ${R}_{1}(I)$ is a set of cofibrations in ${\mathcal{SS}\mathrm{ets}}_{1}$. Lastly, we show that (3) implies (1). Suppose $f$ has the right lifting property with respect to ${R}_{1}(I)$. To prove that $f$ is an acyclic fibration in ${\mathcal{SS}\mathrm{ets}}$, we solve, for $n \geq 0$, a lifting problem
\[\begin{tikzcd}[row sep=3mm, column sep=3mm, ampersand replacement=\&]
	{\partial\Delta[n]} \&\& X \\
	\\
	{\Delta[n]} \&\& {Y.}
	\arrow[from=1-1, to=1-3]
	\arrow["{i_{n}}"', from=1-1, to=3-1]
	\arrow["f"', from=1-3, to=3-3]
	\arrow[dashed, from=3-1, to=1-3]
	\arrow[from=3-1, to=3-3]
\end{tikzcd}\]
Under the adjunction ${R}_{1} \colon {\mathcal{SS}\mathrm{ets}} \rightleftarrows {\mathcal{SS}\mathrm{ets}}_{1} \colon U$ from Proposition \ref{notquite}, we equivalently solve a lifting problem
\[\begin{tikzcd}[row sep=3mm, column sep=3mm, ampersand replacement=\&]
	{{R}_{1}(\partial\Delta[n])} \&\& X \\
	\\
	{{R}_{1}(\Delta[n])} \&\& {Y.}
	\arrow[from=1-1, to=1-3]
	\arrow["{{R}_{1}(i_{n})}"', from=1-1, to=3-1]
	\arrow["f"', from=1-3, to=3-3]
	\arrow[dashed, from=3-1, to=1-3]
	\arrow[from=3-1, to=3-3]
\end{tikzcd}\]
As ${R}_{1}(i_{n})$ is in ${R}_{1}(I)$ and $f$ has the right lifting property with respect to ${R}_{1}(I)$, a solution exists.
\end{proof}
Even though our treatment of Proposition \ref{gencofset} is biased toward the equivalence of statements (2) and (3), which grants us our generating set of cofibrations for ${\mathcal{SS}\mathrm{ets}}_{1}$, we would be remiss if we neglected the equivalence of statements (1) and (2). The fact that statements (1) and (2) are equivalent means that ${\mathcal{SS}\mathrm{ets}}_{1}$ has no pathological acyclic fibrations, which follows from the fact that ${\mathcal{SS}\mathrm{ets}}_{1}$ has no pathological cofibrations. As we shall see later in this section, Proposition \ref{gencofset} finds no analog in the study of acyclic cofibrations in ${\mathcal{SS}\mathrm{ets}}_{1}$ because ${\mathcal{SS}\mathrm{ets}}_{1}$ has fibrations that are dissimilar to Kan fibrations.

\subsection{Acyclic cofibrations of $1$-reduced simplicial sets}\label{section4.3}
Our attempt to emulate Proposition \ref{gencofset} toward a generating set of acyclic cofibrations for ${\mathcal{SS}\mathrm{ets}}_{1}$ encounters a problem: the $1$-reduction of a boundary inclusion is a cofibration in ${\mathcal{SS}\mathrm{ets}}_{1}$, but the $1$-reduction of a horn inclusion need not be an acyclic cofibration in ${\mathcal{SS}\mathrm{ets}}_{1}$.

\begin{lemma}\label{combinebadnews} For $n \geq 1$ and $0 \leq k \leq n$, let $j_{n, k} \colon \Lambda[n, k] \rightarrow \Delta[n]$ denote the horn inclusion in ${\mathcal{SS}\mathrm{ets}}$. Then, ${R}_{1}(j_{n, k})$ is an acyclic cofibration in ${\mathcal{SS}\mathrm{ets}}_{1}$ if and only if $n \neq 2$.
\end{lemma}

\begin{proof} The claim follows from Lemmas \ref{sadcarefree} and \ref{notallislost} and the definition of acyclic cofibrations in ${\mathcal{SS}\mathrm{ets}}_{1}$.
\end{proof}
We can partially work our way around this obstacle. Let $J=\{j_{n, k} \colon \Lambda[n, k] \rightarrow \Delta[n] \mid n \geq 1 \textrm{ and } 0 \leq k \leq n\}$ denote the generating set of acyclic cofibrations for ${\mathcal{SS}\mathrm{ets}}$. Write \[{R}_{1}(J)=\{{R}_{1}(j_{n, k}) \colon {R}_{1}(\Lambda[n, k]) \rightarrow {R}_{1}(\Delta[n]) \mid n \geq 1 \textrm{ and } 0 \leq k \leq n\}.\] To extract a candidate generating set of acyclic cofibrations for ${\mathcal{SS}\mathrm{ets}}$, we use Lemma \ref{sadcarefree} to strike out

\begin{enumerate}

\item the maps ${R}_{1}(j_{1, 0})$ and ${R}_{1}(j_{1, 1})$, which are the identity map of the point; and

\item the maps ${R}_{1}(j_{2, 0})$, ${R}_{1}(j_{2, 1})$, and ${R}_{1}(j_{2, 2})$ which are not weak equivalences.

\end{enumerate}
Therefore, we write $\widetilde{J}=\{j_{n, k} \colon \Lambda[n, k] \rightarrow \Delta[n] \mid n \geq 3 \textrm{ and } 0 \leq k \leq n\}$, whose elements are acyclic cofibrations in ${\mathcal{SS}\mathrm{ets}}$, in order to write \[{R}_{1}(\widetilde{J})=\{{R}_{1}(j_{n, k}) \colon {R}_{1}(\Lambda[n, k]) \rightarrow {R}_{1}(\Delta[n]) \mid n \geq 3 \textrm{ and } 0 \leq k \leq n\},\] which is a set of acyclic cofibrations in ${\mathcal{SS}\mathrm{ets}}_{1}$. Hence, having the right lifting property with respect to the set ${R}_{1}(\widetilde{J})$ is necessary for a map of $1$-reduced simplicial sets to be a fibration in ${\mathcal{SS}\mathrm{ets}}_{1}$. We partially address the question of sufficiency; namely, we establish that the set ${R}_{1}(\widetilde{J})$ detects fibrant objects in ${\mathcal{SS}\mathrm{ets}}_{1}$.

\begin{proposition}\label{quasicataddendum} Let $X$ be a $1$-reduced simplicial set. The following statements are equivalent.

\begin{enumerate}

\item The $1$-reduced simplicial set $X$ is a Kan complex.

\item The $1$-reduced simplicial set $X$ is fibrant in ${\mathcal{SS}\mathrm{ets}}_{1}$.

\item The map $X \rightarrow \Delta[0]$ has the right lifting property with respect to the set \[{R}_{1}(\widetilde{J})=\{{R}_{1}(j_{n, k}) \colon {R}_{1}(\Lambda[n, k]) \rightarrow {R}_{1}(\Delta[n]) \mid n \geq 3 \textrm{ and } 0 \leq k \leq n\}.\]

\end{enumerate}
Hence, the set ${R}_{1}(\widetilde{J})$ detects fibrant objects in ${\mathcal{SS}\mathrm{ets}}_{1}$.
\end{proposition}

\begin{proof} Firstly, (1) and (2) are equivalent by the characterization of fibrant objects in ${\mathcal{SS}\mathrm{ets}}_{1}$ in Theorem \ref{rightinduceeyebrow}. If (2) holds, then $X \rightarrow \Delta[0]$ is a fibration in ${\mathcal{SS}\mathrm{ets}}_{1}$, so it has the right lifting property with respect to the set ${R}_{1}(\widetilde{J})$ because every map in ${R}_{1}(\widetilde{J})$ is an acyclic cofibration in ${\mathcal{SS}\mathrm{ets}}_{1}$. Lastly, suppose (3) holds. We prove that (1) holds. For $n \geq 1$ and $0 \leq k \leq n$, a lifting problem
\[\begin{tikzcd}[row sep=3mm, column sep=3mm, ampersand replacement=\&]
	{\Lambda[n, k]} \&\& X \\
	\\
	{\Delta[n]}
	\arrow["h", from=1-1, to=1-3]
	\arrow["{j_{n, k}}"', from=1-1, to=3-1]
	\arrow["g", dashed, from=3-1, to=1-3]
\end{tikzcd}\]
is equivalent, under the adjunction ${R}_{1} \colon {\mathcal{SS}\mathrm{ets}} \rightleftarrows {\mathcal{SS}\mathrm{ets}}_{1} \colon U$ from Proposition \ref{notquite}, to a lifting problem
\[\begin{tikzcd}[row sep=3mm, column sep=3mm, ampersand replacement=\&]
	{{R}_{1}(\Lambda[n, k])} \&\& X \\
	\\
	{{R}_{1}(\Delta[n]).}
	\arrow["\eta", from=1-1, to=1-3]
	\arrow["{{R}_{1}(j_{n, k})}"', from=1-1, to=3-1]
	\arrow["\psi", dashed, from=3-1, to=1-3]
\end{tikzcd}\]
In light of Lemma \ref{sadcarefree}, we split our solution in three separate cases. If $n=1$, then Lemma \ref{sadcarefree}(1) says that ${R}_{1}(j_{1, k})$ is the identity map of the point, so our lifting problem can be solved by setting $\psi=\eta$. If $n \geq 3$, then our lifting problem can be solved as we are assuming that $X \rightarrow \Delta[0]$ has the right lifting property with respect to the set ${R}_{1}(\widetilde{J})$. Lastly, if $n=2$, then Lemma \ref{sadcarefree}(2) says that our lifting problem is
\[\begin{tikzcd}[row sep=3mm, column sep=3mm, ampersand replacement=\&]
	{\Delta[0]} \&\& X \\
	\\
	{S_{\mathrm{simp}}^{2},}
	\arrow["{x_{0}}", from=1-1, to=1-3]
	\arrow["{{R}_{1}(j_{2, k})}"', from=1-1, to=3-1]
	\arrow["\psi", dashed, from=3-1, to=1-3]
\end{tikzcd}\]
where $x_{0}$ denotes the unique $0$-simplex of $X$, so we take $\psi$ to be the constant map at $x_{0}$.
\end{proof}
In fact, a different and more elaborate argument yields a stronger result.

\begin{proposition}[{\cite[Part II, Proposition 2.12]{Quillen1969}}]\label{quillendidit} Let $f \colon X \rightarrow Y$ be a map of $1$-reduced simplicial sets. Suppose that $Y$ is a Kan complex. Then, $f$ is a fibration in ${\mathcal{SS}\mathrm{ets}}_{1}$ if and only if $f$ has the right lifting property with respect to the set ${R}_{1}(\widetilde{J})$. Hence, the set ${R}_{1}(\widetilde{J})$ detects fibrations in ${\mathcal{SS}\mathrm{ets}}_{1}$ whose target is a Kan complex.
\end{proposition}
Thus, there exist no pathological fibrations in ${\mathcal{SS}\mathrm{ets}}_{1}$ so long as the target is a Kan complex. The additional fibrancy condition on the target is crucially used in Quillen's proof, which suggests that ${R}_{1}(\widetilde{J})$ is not a generating set of acyclic cofibrations for ${\mathcal{SS}\mathrm{ets}}_{1}$. In his discussion of \cite[Part II, Proposition 2.12]{Quillen1969}, Quillen affirms this suggestion and provides a hint toward showing that the fibrancy condition on the target in Proposition \ref{quillendidit} is necessary. However, to the best of our knowledge, the question of whether ${R}_{1}(\widetilde{J})$ is a generating set of acyclic cofibrations for ${\mathcal{SS}\mathrm{ets}}_{1}$ remains open. Our various attempts at producing a map in ${\mathcal{SS}\mathrm{ets}}_{1}$ that has the right lifting property with respect to ${R}_{1}(\widetilde{J})$ but is not a fibration in ${\mathcal{SS}\mathrm{ets}}_{1}$ have been unsuccessful in the same fashion: our candidate maps are not fibrations in ${\mathcal{SS}\mathrm{ets}}_{1}$, but they also fail to have the right lifting property with respect to ${R}_{1}(\widetilde{J})$. We pose the following conjecture to the community.

\begin{conjecture}\label{nooneknows} Let ${R}_{1} \colon {\mathcal{SS}\mathrm{ets}} \rightarrow {\mathcal{SS}\mathrm{ets}}_{1}$ denote the $1$-reduction functor. The set \[{R}_{1}(\widetilde{J})=\{{R}_{1}(j_{n, k}) \colon {R}_{1}(\Lambda[n, k]) \rightarrow {R}_{1}(\Delta[n]) \mid n \geq 3 \textrm{ and } 0 \leq k \leq n\}\] is not a generating set of acyclic cofibrations for ${\mathcal{SS}\mathrm{ets}}_{1}$.
\end{conjecture}

\subsection{Pathological fibrations of $1$-reduced simplicial sets}\label{section4.4}
The issues that we ran into in our study of acyclic cofibrations in ${\mathcal{SS}\mathrm{ets}}_{1}$ suggest that there exist fibrations in ${\mathcal{SS}\mathrm{ets}}_{1}$ that are dissimilar to Kan fibrations. To investigate this phenomenon, we characterize when a fibration in ${\mathcal{SS}\mathrm{ets}}_{1}$ is a Kan fibration. In fact, we produce two results that help us find an explicit example of a fibration in ${\mathcal{SS}\mathrm{ets}}_{1}$ that is not a Kan fibration. To begin with, we note that every Kan fibration of $1$-reduced simplicial sets is a fibration in ${\mathcal{SS}\mathrm{ets}}_{1}$.

\begin{proposition}\label{kanimpliesfibssetr} If $f$ is a Kan fibration of $1$-reduced simplicial sets, then $f$ is a fibration in ${\mathcal{SS}\mathrm{ets}}_{1}$.
\end{proposition}

\begin{proof} If $f$ is a Kan fibration, then $f$ has the right lifting property with respect to the class of acyclic cofibrations in $\mathcal{SS}\mathrm{ets}$, which contains the class of acyclic cofibrations in ${\mathcal{SS}\mathrm{ets}}_{1}$. Thus, $f$ is a fibration in ${\mathcal{SS}\mathrm{ets}}_{1}$.
\end{proof}
We prove a characterization that detects when a fibration in ${\mathcal{SS}\mathrm{ets}}_{1}$ is a Kan fibration and helps us disprove the converse of Proposition \ref{kanimpliesfibssetr}; our result is the $1$-reduced analog of \cite[Lemma 6.6]{GoerssJardine1999}. Recall from Lemma \ref{sadcarefree} that, for $0 \leq k \leq 2$, the map ${R}_{1}(j_{2, k}) \colon \Delta[0] \rightarrow S_{\mathrm{simp}}^{2}$ picks the unique $0$-simplex of the simplicial $2$-sphere.

\begin{proposition}\label{litmus} Let $f$ be a fibration in ${\mathcal{SS}\mathrm{ets}}_{1}$. For $0 \leq k \leq 2$, let $j_{2, k} \colon \Lambda[2, k] \rightarrow \Delta[2]$ denote the horn inclusion. Then, $f$ is a Kan fibration if and only if $f$ has the right lifting property against ${R}_{1}(j_{2, k})$ for $0 \leq k \leq 2$.
\end{proposition}

\begin{proof} Under the adjunction ${R}_{1} \colon {\mathcal{SS}\mathrm{ets}} \rightleftarrows {\mathcal{SS}\mathrm{ets}}_{1} \colon U$ from Proposition \ref{notquite}, a lifting problem of the type
\begin{equation}\tag{I}
\begin{tikzcd}[row sep=3mm, column sep=3mm, ampersand replacement=\&]
	{\Delta[0]} \&\& X \\
	\\
	{S_{\mathrm{simp}}^{2}} \&\& {Y,}
	\arrow[from=1-1, to=1-3]
	\arrow["{{R}_{1}(j_{2, k})}"', from=1-1, to=3-1]
	\arrow["f"', from=1-3, to=3-3]
	\arrow[dashed, from=3-1, to=1-3]
	\arrow[from=3-1, to=3-3]
\end{tikzcd}
\end{equation}
is equivalent to a lifting problem of the type
\begin{equation}\tag{II}
\begin{tikzcd}[row sep=3mm, column sep=3mm, ampersand replacement=\&]
	{\Lambda[2, k]} \&\& X \\
	\\
	{\Delta[2]} \&\& {Y.}
	\arrow[from=1-1, to=1-3]
	\arrow["{j_{2, k}}"', from=1-1, to=3-1]
	\arrow["f"', from=1-3, to=3-3]
	\arrow[dashed, from=3-1, to=1-3]
	\arrow[from=3-1, to=3-3]
\end{tikzcd}
\end{equation}
If $f$ is a Kan fibration, then every lifting problem of type (II) has a solution. Thus, every lifting problem of type (I) has a solution, and $f$ has the right lifting property with respect to ${R}_{1}(j_{2, k})$ for $0 \leq k \leq 2$.

Conversely, if $f$ has the right lifting property with respect to ${R}_{1}(j_{2, k})$ for $0 \leq k \leq 2$, then every lifting problem of type (II) has a solution. If $n \neq 2$ and $0 \leq t \leq n$, then solving a lifting problem
\[\begin{tikzcd}[row sep=3mm, column sep=3mm, ampersand replacement=\&]
	{\Lambda[n, t]} \&\& X \\
	\\
	{\Delta[n]} \&\& {Y}
	\arrow[from=1-1, to=1-3]
	\arrow["{j_{n, t}}"', from=1-1, to=3-1]
	\arrow["f"', from=1-3, to=3-3]
	\arrow[dashed, from=3-1, to=1-3]
	\arrow[from=3-1, to=3-3]
\end{tikzcd}\]
is equivalent, under the adjunction ${R}_{1} \colon {\mathcal{SS}\mathrm{ets}} \rightleftarrows {\mathcal{SS}\mathrm{ets}}_{1} \colon U$ from Proposition \ref{notquite}, to solving a lifting problem
\[\begin{tikzcd}[row sep=3mm, column sep=3mm, ampersand replacement=\&]
	{{R}_{1}(\Lambda[n, t])} \&\& X \\
	\\
	{{R}_{1}(\Delta[n])} \&\& {Y.}
	\arrow[from=1-1, to=1-3]
	\arrow["{{R}_{1}(j_{n, t})}"', from=1-1, to=3-1]
	\arrow["f"', from=1-3, to=3-3]
	\arrow[dashed, from=3-1, to=1-3]
	\arrow[from=3-1, to=3-3]
\end{tikzcd}\]
By Lemma \ref{combinebadnews}, ${R}_{1}(j_{n, t})$ is an acyclic cofibration in ${\mathcal{SS}\mathrm{ets}}_{1}$ because $n \neq 2$, so our lifting problem has a solution because $f$ is a fibration in ${\mathcal{SS}\mathrm{ets}}_{1}$. Overall, we conclude that $f$ is a Kan fibration.
\end{proof}
Proposition \ref{litmus} helps us find a fibration in ${\mathcal{SS}\mathrm{ets}}_{1}$ that bears no resemblance to a Kan fibration. To do so, we need an auxiliary lemma about maps from a $1$-reduced simplicial set to the simplicial $2$-sphere $S_{\mathrm{simp}}^{2}$.

\begin{lemma}\label{nocontradiction} Let $B$ be a $1$-reduced simplicial set. Then, a map $h \colon B \rightarrow S_{\mathrm{simp}}^{2}$ is constant if and only if $h_{\ast} \colon \pi_{2}(B) \rightarrow \pi_{2}(S_{\mathrm{simp}}^{2})$ is not an epimorphism.
\end{lemma}

\begin{proof} If $h$ is constant, then so is $h_{\ast}$. If $h$ is not constant, then there is a $2$-simplex $\omega$ of $B$ such that $h(\omega)$ is the unique non-degenerate $2$-simplex $\epsilon$ of $S_{\mathrm{simp}}^{2}$. As $B$ is $1$-reduced, Corollary \ref{notquitesimplices} allows us to view $\omega$ as a map $\omega \colon S_{\mathrm{simp}}^{2} \rightarrow B$. Then, we compute that $h_{\ast}([\omega])=[h(\omega)]=[\epsilon]$, which is a generator of $\pi_{2}(S_{\mathrm{simp}}^{2}) \cong \mathbb{Z}$.
\end{proof}

\begin{proposition}\label{litmuscor} For $0 \leq k \leq 2$, the map of $1$-reduced simplicial sets ${R}_{1}(j_{2, k}) \colon \Delta[0] \rightarrow S_{\mathrm{simp}}^{2}$ is a fibration in ${\mathcal{SS}\mathrm{ets}}_{1}$ but not a Kan fibration.
\end{proposition}

\begin{proof} We first show that ${R}_{1}(j_{2, k})$ is a fibration in ${\mathcal{SS}\mathrm{ets}}_{1}$. For an acyclic cofibration $\gamma \colon A \rightarrow B$ in ${\mathcal{SS}\mathrm{ets}}_{1}$ and a lifting problem
\[\begin{tikzcd}[row sep=3mm, column sep=3mm, ampersand replacement=\&]
	A \&\&\& {\Delta[0]} \\
	\\
	B \&\&\& {S_{\mathrm{simp}}^{2},}
	\arrow["c", from=1-1, to=1-4]
	\arrow["\gamma"', from=1-1, to=3-1]
	\arrow["{{R}_{1}(j_{2, k})}", from=1-4, to=3-4]
	\arrow["g", dashed, from=3-1, to=1-4]
	\arrow["h", from=3-1, to=3-4]
\end{tikzcd}\]
we claim that the unique option for $g$ solves our lifting problem. The upper left triangle has to commute. To show that the lower right triangle also commutes, we show that $h$ is constant. By Lemma \ref{nocontradiction}, we show that $h_{\ast} \colon \pi_{2}(B) \rightarrow \pi_{2}(S_{\mathrm{simp}}^{2})$ is not an epimorphism; in fact, we show that $h_{\ast}$ is constant. As $c_{\ast}$ is constant, so is \[h_{\ast}\gamma_{\ast}=(h\gamma)_{\ast}=({{R}_{1}(j_{2, k})}c)_{\ast}={{R}_{1}(j_{2, k})}_{\ast}c_{\ast}.\] Because $\gamma$ is an acyclic cofibration in ${\mathcal{SS}\mathrm{ets}}_{1}$, we know that $\gamma_{\ast}$ is an isomorphism. Therefore, $h_{\ast}$ is constant.

Now, we apply the characterization from Proposition \ref{litmus} to argue that ${R}_{1}(j_{2, k})$ is not a Kan fibration because it does not have the right lifting property with respect to itself. Indeed, the lifting problem
\[\begin{tikzcd}[row sep=3mm, column sep=3mm, ampersand replacement=\&]
	{\Delta[0]} \&\&\& {\Delta[0]} \\
	\\
	{S_{\mathrm{simp}}^{2}} \&\&\& {S_{\mathrm{simp}}^{2}}
	\arrow[Rightarrow, no head, from=1-1, to=1-4]
	\arrow["{{R}_{1}(j_{2, k})}"', from=1-1, to=3-1]
	\arrow["{{R}_{1}(j_{2, k})}", from=1-4, to=3-4]
	\arrow["g", dashed, from=3-1, to=1-4]
	\arrow[Rightarrow, no head, from=3-1, to=3-4]
\end{tikzcd}\]
has no solution because the unique option for $g$ does not make the lower right triangle commute.
\end{proof}
In the presence of a fibrancy condition on the target, Proposition \ref{litmus} can be strengthened. Namely, we establish the following characterization result, which is the $1$-reduced analog of \cite[Corollary 6.9]{GoerssJardine1999}.

\begin{proposition}\label{litmussuper} Let $f \colon X \rightarrow Y$ be a fibration in ${\mathcal{SS}\mathrm{ets}}_{1}$. Suppose that $Y$ is a Kan complex. Then, $X$ is also a Kan complex, and the following statements are equivalent.

\begin{enumerate}

\item The map $f \colon X \rightarrow Y$ is a Kan fibration.

\item The map $f \colon X \rightarrow Y$ has the right lifting property with respect to ${R}_{1}(j_{2, k})$ for $0 \leq k \leq 2$.

\item The map $f \colon X \rightarrow Y$ induces an epimorphism $f_{\ast} \colon \pi_{2}(X) \rightarrow \pi_{2}(Y)$.

\end{enumerate}

\end{proposition}

\begin{proof} The unique map $X \rightarrow \Delta[0]$ factors in ${\mathcal{SS}\mathrm{ets}}_{1}$ as the composite of fibrations $X \xrightarrow{f} Y \rightarrow \Delta[0]$. Hence, $X$ is fibrant in ${\mathcal{SS}\mathrm{ets}}_{1}$, thus a Kan complex by the characterization of fibrant objects in ${\mathcal{SS}\mathrm{ets}}_{1}$ from Theorem \ref{rightinduceeyebrow}. Now, Proposition \ref{litmus} says that (1) and (2) are equivalent. We show that (2) and (3) are equivalent. Firstly, suppose that $f$ has the right lifting property with respect to ${R}_{1}(j_{2, k})$ for $0 \leq k \leq 2$. As $Y$ is a Kan complex, every element of $\pi_{2}(Y)$ is of the form $[\omega]$ for a map $\omega \colon S_{\mathrm{simp}}^{2} \rightarrow Y$. We solve the lifting problem
\[\begin{tikzcd}[row sep=3mm, column sep=3mm, ampersand replacement=\&]
	{\Delta[0]} \&\& X \\
	\\
	{S_{\mathrm{simp}}^{2}} \&\& Y
	\arrow[from=1-1, to=1-3]
	\arrow[from=1-1, to=3-1]
	\arrow["f"', from=1-3, to=3-3]
	\arrow["\alpha", dashed, from=3-1, to=1-3]
	\arrow["\omega", from=3-1, to=3-3]
\end{tikzcd}\]
to produce an element $[\alpha] \in \pi_{2}(X)$ such that $f_{\ast}([\alpha])=[f\alpha]=[\omega]$.

Conversely, suppose that $f$ induces an epimorphism at $\pi_{2}$. For $0 \leq k \leq 2$, we solve a lifting problem
\[\begin{tikzcd}[row sep=3mm, column sep=3mm, ampersand replacement=\&]
	{\Delta[0]} \&\& X \\
	\\
	{S_{\mathrm{simp}}^{2}} \&\& {Y.}
	\arrow[from=1-1, to=1-3]
	\arrow["{{R}_{1}(j_{2, k})}"', from=1-1, to=3-1]
	\arrow["f"', from=1-3, to=3-3]
	\arrow["\psi", dashed, from=3-1, to=1-3]
	\arrow["\eta", from=3-1, to=3-3]
\end{tikzcd}\]
As $f_{\ast}$ is an epimorphism, there is a commutative diagram
\[\begin{tikzcd}[row sep=3mm, column sep=3mm, ampersand replacement=\&]
	\&\& {S_{\mathrm{simp}}^{2}} \&\& X \\
	\\
	{S_{\mathrm{simp}}^{2}} \&\& {S_{\mathrm{simp}}^{2} \wedge \Delta[1]_{+}} \&\& {Y,}
	\arrow["\beta", from=1-3, to=1-5]
	\arrow["{d^{1}}"', from=1-3, to=3-3]
	\arrow["f"', from=1-5, to=3-5]
	\arrow["{d^{0}}", from=3-1, to=3-3]
	\arrow["\eta"', curve={height=24pt}, from=3-1, to=3-5]
	\arrow["H", from=3-3, to=3-5]
\end{tikzcd}\]
where $f_{\ast}([\beta])=[f\beta]=[\eta]$ by means of the homotopy $H \colon S_{\mathrm{simp}}^{2} \wedge \Delta[1]_{+} \rightarrow Y$ starting from $Hd^{1}=f\beta$ and ending at $Hd^{0}=\eta$. By Lemma \ref{smashingtensor}, $S_{\mathrm{simp}}^{2} \wedge \Delta[1]_{+}$ is $1$-reduced. As $d^{1}$ is an acyclic cofibration in ${\mathcal{SS}\mathrm{ets}}_{1}$ and $f$ is a fibration in ${\mathcal{SS}\mathrm{ets}}_{1}$, there exists a lift
\[\begin{tikzcd}[row sep=3mm, column sep=3mm, ampersand replacement=\&]
	\&\& {S_{\mathrm{simp}}^{2}} \&\& X \\
	\\
	{S_{\mathrm{simp}}^{2}} \&\& {S_{\mathrm{simp}}^{2} \wedge \Delta[1]_{+}} \&\& {Y.}
	\arrow["\beta", from=1-3, to=1-5]
	\arrow["{d^{1}}"', from=1-3, to=3-3]
	\arrow["f"', from=1-5, to=3-5]
	\arrow["{d^{0}}", from=3-1, to=3-3]
	\arrow["\eta"', curve={height=24pt}, from=3-1, to=3-5]
	\arrow["g", dashed, from=3-3, to=1-5]
	\arrow["H", from=3-3, to=3-5]
\end{tikzcd}\]
Then, $\psi=gd^{0}$ solves our lifting problem because $f\psi=fgd^{0}=Hd^{0}=\eta$.
\end{proof}
The fibrancy condition in Proposition \ref{litmussuper} is the reason why we could not have used it to streamline our proof of Proposition \ref{litmuscor}. Indeed, we conclude this section by showing that, for $n \geq 2$, the simplicial $n$-sphere $S_{\mathrm{simp}}^{n}$ is not a Kan complex. To that end, we need a counting lemma.

\begin{lemma}\label{countingfaces} Let $x$ be an $n$-simplex of a simplicial set. For $0 \leq k \leq n$, let $s_{k}(x)$ be the $k$-th degeneracy of $x$.

\begin{enumerate}

\item If $x$ is a non-degenerate $n$-simplex, then $s_{k}(x)$ has exactly $2$ non-degenerate faces.

\item If $x$ is a degenerate $n$-simplex, then $s_{k}(x)$ has no non-degenerate faces.

\end{enumerate}
Thus, the degenerate $(n+1)$-simplex $s_{k}(x)$ has at most $2$ non-degenerate faces.

\end{lemma}

\begin{proof} For $0 \leq i \leq n+1$, let $d_{i}s_{k}(x)$ be the $i$-th face of $s_{k}(x)$. The result follows from the simplicial identities \[d_{i}s_{k}(x) = \begin{cases} s_{k-1}d_{i}(x), \textrm{ if } 0 \leq i \leq k-1; \\
x, \textrm{ if } k \leq i \leq k+1; \\
s_{k}d_{i-1}(x), \textrm{ if } k+2 \leq i \leq n+1. \end{cases}\]
\end{proof}

Our counting lemma allows us to prove that, for $n \geq 2$, the simplicial $n$-sphere $S_{\mathrm{simp}}^{n}$ is not a Kan complex.

\begin{proposition}\label{srkannot} For $n \geq 2$, the simplicial $n$-sphere $S_{\mathrm{simp}}^{n}$ is not a Kan complex, thus not fibrant in ${\mathcal{SS}\mathrm{ets}}_{1}$.
\end{proposition}

\begin{proof} Write $\epsilon$ for the unique non-degenerate $n$-simplex of $S_{\mathrm{simp}}^{n}$. For $0 \leq k \leq n+1$, consider the lifting problem
\[\begin{tikzcd}[row sep=3mm, column sep=3mm, ampersand replacement=\&]
	{\Lambda[n+1, k]} \&\& {S_{\mathrm{simp}}^{n}} \\
	\\
	{\Delta[n+1],}
	\arrow["h", from=1-1, to=1-3]
	\arrow["{j_{n+1, k}}"', from=1-1, to=3-1]
	\arrow["g", dashed, from=3-1, to=1-3]
\end{tikzcd}\]
where $h$ is the horn that picks out the $n$-simplices $(\epsilon, \dots, \epsilon)$. Any simplex $g$ that fills in this horn would be a degenerate $(n+1)$-simplex with $n+1$ non-degenerate faces, which is impossible for $n \geq 2$ by Lemma \ref{countingfaces}.
\end{proof}

\section{The simplicial model structure on $1$-reduced simplicial sets}\label{section5}
A convenient feature of the model categories ${\mathcal{SS}\mathrm{ets}}$ and ${\mathcal{SS}\mathrm{ets}}_{\ast}$ is that they are simplicial, so they carry more structure for us to use. In particular, between any two objects, we have an explicit construction of a function complex, which facilitates the study of left Bousfield localizations of these model categories.

In this section, we endow ${\mathcal{SS}\mathrm{ets}}_{1}$ with a simplicial model structure. This model structure makes our subsequent left Bousfield localization of ${\mathcal{SS}\mathrm{ets}}_{1}$ more convenient in practice. Moreover, as left Bousfield localizations preserve the simplicial structure of a model category, the entire family of model structures on $1$-reduced simplicial sets inherits the simplicial structure in this section, so we get significant traction out of our argument.

\subsection{Simplicial model categories}\label{section5.1}
A reference on simplicial model categories is \cite[Chapter 9]{Hirschhorn2003}. We reformulate axiom (2) in the definition from {\cite[Definition 9.1.2]{Hirschhorn2003}} for our purposes using {\cite[Proposition 9.3.7]{Hirschhorn2003}}.

\begin{definition}\label{defsmcdef} Let $\mathcal{M}$ be a model category that is enriched in simplicial sets; in particular, for two objects $X$ and $Y$ of $\mathcal{M}$, we have a \emph{function complex} $\operatorname{Map}_{\mathcal{M}}(X, Y)$. Then, $\mathcal{M}$ is a \emph{simplicial model category} if $\mathcal{M}$ satisfies the following axioms.

\begin{enumerate}

\item (\emph{Adjointness axiom}) For two objects $X$ and $Y$ of $\mathcal{M}$ and a simplicial set $Z$, there is a \emph{tensor} object $X \otimes Z$ of $\mathcal{M}$ and a \emph{cotensor} object $Y^{Z}$ of $\mathcal{M}$ satisfying the natural isomorphisms of simplicial sets
\[\operatorname{Map}_{\mathcal{M}}(X \otimes Z, Y) \cong \operatorname{Map}_{{\mathcal{SS}\mathrm{ets}}}(Z, \operatorname{Map}_{\mathcal{M}}(X, Y)) \cong \operatorname{Map}_{\mathcal{M}}(X, Y^{Z}).\]
Note that $\operatorname{Map}_{{\mathcal{SS}\mathrm{ets}}}(Z, \operatorname{Map}_{\mathcal{M}}(X, Y))$ is a function complex of simplicial sets; see Example \ref{defsmcdef1}.

\item (\emph{Homotopy lifting-extension axiom}) If $i \colon A \rightarrow B$ is a cofibration in $\mathcal{M}$ and $j \colon L \rightarrow K$ is a cofibration of simplicial sets, then the induced morphism
\[A \otimes K \underset{A \otimes L}{\coprod} B \otimes L \rightarrow B \otimes K\]
is a cofibration in $\mathcal{M}$ that is acyclic if $i$ or $j$ is acyclic.

\end{enumerate}

\end{definition}

We present the simplicial structure of the model categories ${\mathcal{SS}\mathrm{ets}}$ and ${\mathcal{SS}\mathrm{ets}}_{\ast}$.

\begin{example}[{\cite[Example 9.1.13]{Hirschhorn2003}}]\label{defsmcdef1} The category ${\mathcal{SS}\mathrm{ets}}$ is a simplicial model category. Given simplicial sets $X$, $Y$, and $Z$, we define

\begin{enumerate}

\item the function complex $\operatorname{Map}_{{\mathcal{SS}\mathrm{ets}}}(X, Y)$ to be the simplicial set with $n$-simplices $\operatorname{Hom}_{{\mathcal{SS}\mathrm{ets}}}(X \times \Delta[n], Y)$;

\item the tensor of $X$ with $Z$ to be the simplicial set $X \times Z$; and

\item the cotensor of $X$ with $Z$ to be the simplicial set $\operatorname{Map}_{{\mathcal{SS}\mathrm{ets}}}(Z, X)$.

\end{enumerate}

\end{example}

\begin{example}[{\cite[Example 9.1.14]{Hirschhorn2003}}]\label{defsmcdef2} The category ${\mathcal{SS}\mathrm{ets}}_{\ast}$ is a simplicial model category. For two pointed simplicial sets $X$ and $Y$ and a simplicial set $Z$, we define

\begin{enumerate}

\item the function complex $\operatorname{Map}_{{\mathcal{SS}\mathrm{ets}}_{\ast}}(X, Y)$ as the simplicial set with $n$-simplices $\operatorname{Hom}_{{\mathcal{SS}\mathrm{ets}}_{\ast}}(X \wedge \Delta[n]_{+}, Y)$;

\item the tensor of $X$ with $Z$ to be the pointed simplicial set $X \wedge Z_{+}$; and

\item the cotensor of $X$ with $Z$ to be the pointed simplicial set $\operatorname{Map}_{{\mathcal{SS}\mathrm{ets}}_{\ast}}(Z_{+}, X)$.

\end{enumerate}

\end{example}

\subsection{The simplicial model structure}\label{section5.2} As our left transfer reveals to us, the homotopy theory of $1$-reduced simplicial sets is closer to that of pointed simplicial sets. Thus, it is the simplicial structure from Example \ref{defsmcdef2} that we modify to get a simplicial structure on ${\mathcal{SS}\mathrm{ets}}_{1}$. In fact, our modification guarantees that the homotopy lifting-extension axiom for ${\mathcal{SS}\mathrm{ets}}_{1}$ is identical to that for ${\mathcal{SS}\mathrm{ets}}_{\ast}$, so we circumvent what is often the most difficult part of endowing a model category with a simplicial structure.

For the simplicial structure on the model category ${\mathcal{SS}\mathrm{ets}}_{1}$, we define

\begin{enumerate}

\item the function complex of two $1$-reduced simplicial sets $X$ and $Y$ to be $\operatorname{Map}_{{\mathcal{SS}\mathrm{ets}}_{\ast}}(X, Y)$;

\item the tensor of a $1$-reduced simplicial set $X$ with a simplicial set $Z$ to be $X \wedge Z_{+}$; and

\item the cotensor of a $1$-reduced simplicial set $X$ with a simplicial set $Z$ to be $E_{1}(\operatorname{Map}_{{\mathcal{SS}\mathrm{ets}}_{\ast}}(Z_{+}, X))$.

\end{enumerate}
With our definition of function complexes, ${\mathcal{SS}\mathrm{ets}}_{1}$ is enriched in simplicial sets by Example \ref{defsmcdef2}. The homotopy lifting-extension axiom also holds by Example \ref{defsmcdef2}, so it remains to verify the adjointness axiom.

\begin{proposition}\label{simplicialaxiomm6} Given two $1$-reduced simplicial sets $X$ and $Y$ and a simplicial set $Z$, there are isomorphisms of simplicial sets
\[\operatorname{Map}_{{\mathcal{SS}\mathrm{ets}}_{\ast}}(X \wedge Z_{+}, Y) \cong \operatorname{Map}_{{\mathcal{SS}\mathrm{ets}}}(Z, \operatorname{Map}_{{\mathcal{SS}\mathrm{ets}}_{\ast}}(X, Y)) \cong \operatorname{Map}_{{\mathcal{SS}\mathrm{ets}}_{\ast}}(X, E_{1}(\operatorname{Map}_{{\mathcal{SS}\mathrm{ets}}_{\ast}}(Z_{+}, Y)))\]
that are natural in $X$, $Y$, and $Z$.
\end{proposition}

\begin{proof} The natural isomorphism $\operatorname{Map}_{{\mathcal{SS}\mathrm{ets}}_{\ast}}(X \wedge Z_{+}, Y) \cong \operatorname{Map}_{{\mathcal{SS}\mathrm{ets}}}(Z, \operatorname{Map}_{{\mathcal{SS}\mathrm{ets}}_{\ast}}(X, Y))$ is a consequence of Example \ref{defsmcdef2}. To prove that $\operatorname{Map}_{{\mathcal{SS}\mathrm{ets}}_{\ast}}(X \wedge Z_{+}, Y) \cong \operatorname{Map}_{{\mathcal{SS}\mathrm{ets}}_{\ast}}(X, E_{1}(\operatorname{Map}_{{\mathcal{SS}\mathrm{ets}}_{\ast}}(Z_{+}, Y)))$, we compute that

\begin{flalign*}
\operatorname{Hom}_{{\mathcal{SS}\mathrm{ets}}_{\ast}}((X \wedge Z_{+}) \wedge \Delta[n]_{+}, Y) &\cong \operatorname{Hom}_{{\mathcal{SS}\mathrm{ets}}_{\ast}}(X \wedge \Delta[n]_{+}, \operatorname{Map}_{{\mathcal{SS}\mathrm{ets}}_{\ast}}(Z_{+}, Y)) \\
&\cong \operatorname{Hom}_{{\mathcal{SS}\mathrm{ets}}_{1}}(X \wedge \Delta[n]_{+}, E_{1}(\operatorname{Map}_{{\mathcal{SS}\mathrm{ets}}_{\ast}}(Z_{+}, Y))) \\
 &= \operatorname{Hom}_{{\mathcal{SS}\mathrm{ets}}_{\ast}}(X \wedge \Delta[n]_{+}, E_{1}(\operatorname{Map}_{{\mathcal{SS}\mathrm{ets}}_{\ast}}(Z_{+}, Y))),
\end{flalign*}
where the first bijection comes from Example \ref{defsmcdef2} and the second bijection follows from Proposition \ref{yesquite}(2).
\end{proof}

Overall, our modification of Example \ref{defsmcdef2} endows the model category ${\mathcal{SS}\mathrm{ets}}_{1}$ with a simplicial structure.

\begin{theorem}\label{simplicialdefsworks} Let ${\mathcal{SS}\mathrm{ets}}_{1}$ denote the model category from Theorem \ref{rightinduceeyebrow}. We define

\begin{enumerate}

\item the function complex of two $1$-reduced simplicial sets $X$ and $Y$ to be $\operatorname{Map}_{{\mathcal{SS}\mathrm{ets}}_{\ast}}(X, Y)$;

\item the tensor of a $1$-reduced simplicial set $X$ with a simplicial set $Z$ to be $X \wedge Z_{+}$; and

\item the cotensor of a $1$-reduced simplicial set $X$ with a simplicial set $Z$ to be $E_{1}(\operatorname{Map}_{{\mathcal{SS}\mathrm{ets}}_{\ast}}(Z_{+}, X))$.

\end{enumerate}
With these definitions, ${\mathcal{SS}\mathrm{ets}}_{1}$ has the structure of a simplicial model category.
\end{theorem}

Lastly, we note that, in \cite{Burke2021}, Burke establishes an analogous simplicial structure for \emph{reduced simplicial sets}, that is, simplicial sets with a unique $0$-simplex; Burke's line of argument is different from ours, though.

\section{Rational homotopy theory of $1$-connected spaces}\label{section6}
From now on, let $M$ be a multiplicative subset of $\mathbb{Z}$. We gather the background and auxiliary results from the rational homotopy theory of $1$-connected topological spaces that we need to recover the entirety of Quillen's family of model structures on ${\mathcal{SS}\mathrm{ets}}_{1}$. In particular, we describe the $M$-local spaces and $M$-local homotopy equivalences in terms of function complexes, which prepares us for the left Bousfield localization in Section \ref{section7}.

\subsection{Rationalization of $1$-connected spaces}\label{section6.1}
We give a review from \cite[Chapter 2]{Sullivan1970} and \cite[Chapter 9]{FHT2001}.

\begin{definition}\label{defrhtdef} Let $M$ be a multiplicative subset of $\mathbb{Z}$.

\begin{enumerate}

\item A map of $1$-connected spaces $f \colon X \rightarrow Y$ is an \emph{$M$-local homotopy equivalence} if, for every $n \geq 2$, the map $f_{\ast} \otimes \operatorname{id}_{M^{-1}\mathbb{Z}} \colon \pi_{n}(X) \otimes M^{-1}\mathbb{Z} \rightarrow \pi_{n}(Y) \otimes M^{-1}\mathbb{Z}$ is an isomorphism.

\item A $1$-connected space $Y$ is \emph{$M$-local} if, for every $n \geq 2$, the abelian group $\pi_{n}(Y)$ is an $M^{-1}\mathbb{Z}$-module.

\item An \emph{$M$-localization} of a $1$-connected space $Y$ is an $M$-local homotopy equivalence with $M$-local target $f \colon Y \rightarrow Y_{M^{-1}\mathbb{Z}}$.

\end{enumerate}
When $M^{-1}\mathbb{Z}=\mathbb{Q}$, we speak of a \emph{rational homotopy equivalence}, a \emph{rational space}, and a \emph{rationalization}.

\end{definition}

\begin{remark}\label{defrhtdef1} If $f \colon Y \rightarrow Y_{\mathbb{Q}}$ is a rationalization of $Y$, then, for every $n \geq 2$, $f$ induces the isomorphism $f_{\ast} \otimes \operatorname{id}_{\mathbb{Q}} \colon \pi_{n}(Y) \otimes \mathbb{Q} \xrightarrow{\cong} \pi_{n}(Y_{\mathbb{Q}})$ that discards torsion in the $n$-th homotopy group of $Y$ in a topological fashion.
\end{remark}

The following example features often in our arguments.

\begin{example}\label{defrhtdef2} Every weak homotopy equivalence of $1$-connected spaces is an $M$-local homotopy equivalence. Also, every $M$-local homotopy equivalence of $M$-local spaces is a weak homotopy equivalence.
\end{example}

What is more, the rationalization of the $n$-sphere in dimensions $n \geq 2$ is a building block for the rationalization of all $1$-connected spaces; as such, it is a cornerstone of our arguments in the remainder of this paper. We do not need the details of the geometric construction at the moment, though we present them in Section \ref{section9.2}.

\begin{example}\label{defrhtdef3} For $n \geq 2$, the \emph{$M$-local $n$-sphere} $S^{n}_{M^{-1}\mathbb{Z}}$ is an $M$-local CW complex containing the $n$-sphere $S^{n}$ as a CW subcomplex, and the inclusion $S^{n} \rightarrow S^{n}_{M^{-1}\mathbb{Z}}$ is an $M$-localization of $S^{n}$.
\end{example}
Every $1$-connected space can be $M$-localized, and thus rationalized, uniquely up to homotopy equivalence.

\begin{theorem}[{\cite[Theorem 9.7]{FHT2001}}]\label{thebigdeal} Let $M$ be a multiplicative subset of $\mathbb{Z}$. Let $Y$ be a $1$-connected space.

\begin{enumerate}

\item There exists a relative CW complex $(Y_{M^{-1}\mathbb{Z}}, Y)$ with no $0$-cells and no $1$-cells such that the inclusion $Y \rightarrow Y_{M^{-1}\mathbb{Z}}$ is an $M$-localization of $Y$.

\item Every map $f \colon Y \rightarrow Z$ with $M$-local target extends to a map $g \colon Y_{M^{-1}\mathbb{Z}} \rightarrow Z$ such that the diagram

\[\begin{tikzcd}[row sep=3mm, column sep=3mm, ampersand replacement=\&]
	Y \&\& {Y_{M^{-1}\mathbb{Z}}} \\
	\\
	\&\& {Z}
	\arrow[from=1-1, to=1-3]
	\arrow["f"', from=1-1, to=3-3]
	\arrow["g"', dashed, from=1-3, to=3-3]
\end{tikzcd}\]
commutes, and every homotopy $f \simeq f'$ extends to a homotopy $g \simeq g'$. Thus, every map $f \colon Y \rightarrow Z$ with $M$-local target extends to a map $g \colon Y_{M^{-1}\mathbb{Z}} \rightarrow Z$ uniquely up to homotopy.

\end{enumerate}
In particular, any other space $Y'_{M^{-1}\mathbb{Z}}$ satisfying (1) and (2) is homotopy equivalent (relative to $Y$) to $Y_{M^{-1}\mathbb{Z}}$.
\end{theorem}

\begin{remark}\label{cwbusiness} For future reference, note that, if $Y$ is a $1$-connected CW complex, then so is $Y_{M^{-1}\mathbb{Z}}$.
\end{remark}

\subsection{A characterization of rational spaces}\label{section6.2}
For a multiplicative subset $M$ of $\mathbb{Z}$, we characterize $M$-local spaces, and thus rational spaces, in a way that supports our left Bousfield localization in Section \ref{section7}. Namely, we describe when a $1$-connected space is $M$-local in the language of function complexes.

For two topological spaces $X$ and $Y$, the {function complex} $\operatorname{Map}_{\mathcal{T}\mathrm{op}}(X, Y)$ is the simplicial set with $n$-simplices $\operatorname{Hom}_{\mathcal{T}\mathrm{op}}(X \times |\Delta[n]|, Y)$. Observe that \[\pi_{0}(\operatorname{Map}_{\mathcal{T}\mathrm{op}}(X, Y))=[X, Y]\] is the set of homotopy classes of maps from $X$ to $Y$. The following lemma is a straightforward application of the universal property of $M$-localization.

\begin{lemma}\label{ineedthisonetoo} Let $M$ be a multiplicative subset of $\mathbb{Z}$. Let $X$ and $Y$ be $1$-connected spaces. If $Y$ is $M$-local, then the restriction map $\operatorname{Map}_{\mathcal{T}\mathrm{op}}(X_{M^{-1}\mathbb{Z}}, Y) \rightarrow \operatorname{Map}_{\mathcal{T}\mathrm{op}}(X, Y)$ is a weak equivalence.
\end{lemma}
We are now ready to characterize $M$-local spaces, and thus rational spaces, using function complexes.

\begin{proposition}\label{illdothisone} Let $M$ be a multiplicative subset of $\mathbb{Z}$. A $1$-connected space $Y$ is $M$-local if and only if, for every $n \geq 2$, the restriction map $\operatorname{Map}_{\mathcal{T}\mathrm{op}}(S^{n}_{M^{-1}\mathbb{Z}}, Y) \rightarrow \operatorname{Map}_{\mathcal{T}\mathrm{op}}(S^{n}, Y)$ is a weak equivalence.
\end{proposition}

\begin{proof} If $Y$ is $M$-local, then Lemma \ref{ineedthisonetoo} implies that the restriction map is a weak equivalence. Conversely, suppose that, for $n \geq 2$, the map $\operatorname{Map}_{\mathcal{T}\mathrm{op}}(S^{n}_{M^{-1}\mathbb{Z}}, Y) \rightarrow \operatorname{Map}_{\mathcal{T}\mathrm{op}}(S^{n}, Y)$ is a weak equivalence. For $n \geq 2$, let $[f]$ be an element of $\pi_{n}(Y)$ represented by a map $f \colon S^{n} \rightarrow Y$, and let $m \in M$. Then, there exists a map $g \colon S^{n}_{M^{-1}\mathbb{Z}} \rightarrow Y$, unique up to homotopy, such that the diagram
\[\begin{tikzcd}[row sep=3mm, column sep=3mm, ampersand replacement=\&]
	{S^{n}} \&\& {S^{n}_{M^{-1}\mathbb{Z}}} \\
	\\
	\&\& Y
	\arrow["{\tau}", from=1-1, to=1-3]
	\arrow["f"', from=1-1, to=3-3]
	\arrow["g"', from=1-3, to=3-3]
\end{tikzcd}\]
commutes up to homotopy, where $\tau$ denotes the inclusion. In particular, homotopy commutativity implies that \[[f]=f_{\ast}([\operatorname{id}_{S^{n}}])=g_{\ast}\tau_{\ast}([\operatorname{id}_{S^{n}}])=g_{\ast}([\tau]).\] As $S^{n}_{M^{-1}\mathbb{Z}}$ is $M$-local, there exists a unique element $\dfrac{1}{m} \cdot [\tau]$, so we uniquely define $\dfrac{1}{m} \cdot [f]=g_{\ast}\left(\dfrac{1}{m} \cdot [\tau]\right)$.
\end{proof}
As our function complexes of $1$-reduced simplicial sets are pointed, we need the pointed version of Proposition \ref{illdothisone}. We write $\mathcal{T}\mathrm{op}_{\ast}$ for the category of pointed spaces. For two pointed spaces $X$ and $Y$, the function complex $\operatorname{Map}_{\mathcal{T}\mathrm{op}_{\ast}}(X, Y)$ is the simplicial set with $n$-simplices $\operatorname{Map}_{\mathcal{T}\mathrm{op}_{\ast}}(X, Y)_{n}=\operatorname{Hom}_{\mathcal{T}\mathrm{op}_{\ast}}(X \wedge |\Delta[n]|_{+}, Y)$. Again, $\pi_{0}(\operatorname{Map}_{\mathcal{T}\mathrm{op}_{\ast}}(X, Y))=[X, Y]_{\ast}$ is the set of pointed homotopy classes of pointed maps from $X$ to $Y$. We state the pointed version of Proposition \ref{illdothisone}, which is proven identically.

\begin{proposition}\label{illdothisonept} Let $M$ be a multiplicative subset of $\mathbb{Z}$. A $1$-connected pointed space $Y$ is $M$-local if and only if, for every $n \geq 2$, the restriction map $\operatorname{Map}_{\mathcal{T}\mathrm{op}_{\ast}}(S^{n}_{M^{-1}\mathbb{Z}}, Y) \rightarrow \operatorname{Map}_{\mathcal{T}\mathrm{op}_{\ast}}(S^{n}, Y)$ is a weak equivalence.
\end{proposition}

\subsection{A characterization of rational homotopy equivalences}\label{section6.3}
For a multiplicative subset $M$ of $\mathbb{Z}$, we also characterize the $M$-local homotopy equivalences using function complexes in preparation for the left Bousfield localization in Section \ref{section7}. The following lemma, together with its pointed analog, underpins our arguments.

\begin{lemma}[{\cite[Corollary 9.3.3(2)]{Hirschhorn2003}}]\label{citeadnauseam}

\begin{enumerate}

\item If $f \colon X \rightarrow Y$ is a weak equivalence of simplicial sets and $Z$ is a Kan complex, then $\operatorname{Map}_{\mathcal{SS}\mathrm{ets}}(f, Z) \colon \operatorname{Map}_{\mathcal{SS}\mathrm{ets}}(Y, Z) \rightarrow \operatorname{Map}_{\mathcal{SS}\mathrm{ets}}(X, Z)$ is a weak equivalence.

\item If $f \colon X \rightarrow Y$ is a weak homotopy equivalence of CW complexes and $Z$ is a topological space, then $\operatorname{Map}_{\mathcal{T}\mathrm{op}}(f, Z) \colon \operatorname{Map}_{\mathcal{T}\mathrm{op}}(Y, Z) \rightarrow \operatorname{Map}_{\mathcal{T}\mathrm{op}}(X, Z)$ is a weak equivalence.

\end{enumerate}

\end{lemma}
We characterize the $M$-local homotopy equivalences, and thus the rational homotopy equivalences, as follows.

\begin{proposition}\label{illdothisonetoo} Let $M$ be a multiplicative subset of $\mathbb{Z}$. Let $f \colon X \rightarrow Y$ be a map of $1$-connected CW complexes. The following statements are equivalent.

\begin{enumerate}

\item The map $f$ is an $M$-local homotopy equivalence.

\item The induced map on $M$-localizations $f_{M^{-1}\mathbb{Z}} \colon X_{M^{-1}\mathbb{Z}} \rightarrow Y_{M^{-1}\mathbb{Z}}$ is a weak homotopy equivalence.

\item The induced map on $M$-localizations $f_{M^{-1}\mathbb{Z}} \colon X_{M^{-1}\mathbb{Z}} \rightarrow Y_{M^{-1}\mathbb{Z}}$ is a homotopy equivalence.

\item For every $M$-local space $Z$, the map $\operatorname{Map}_{\mathcal{T}\mathrm{op}}(f_{M^{-1}\mathbb{Z}}, Z)$ is a weak equivalence.

\item For every $M$-local space $Z$, the map $\operatorname{Map}_{\mathcal{T}\mathrm{op}}(f, Z)$ is a weak equivalence.

\end{enumerate}

\end{proposition}

\begin{proof} Firstly, (1) and (2) are equivalent by applying Example \ref{defrhtdef2} to the commutative diagram
\[\begin{tikzcd}[row sep=3mm, column sep=3mm, ampersand replacement=\&]
	X \&\& Y \\
	\\
	{X_{M^{-1}\mathbb{Z}}} \&\& {Y_{M^{-1}\mathbb{Z}}}
	\arrow[from=1-1, to=1-3]
	\arrow[from=1-1, to=3-1]
	\arrow[from=1-3, to=3-3]
	\arrow[from=3-1, to=3-3]
\end{tikzcd}\]
whose vertical maps are $M$-local homotopy equivalences. Then, (2) and (3) are equivalent by the Whitehead theorem, as both $X_{M^{-1}\mathbb{Z}}$ and $Y_{M^{-1}\mathbb{Z}}$ are CW complexes by Remark \ref{cwbusiness}. Moreover, (4) and (5) are equivalent because, for an $M$-local space $Z$, we have the commutative diagram
\[\begin{tikzcd}[row sep=3mm, column sep=3mm, ampersand replacement=\&]
	{\operatorname{Map}_{\mathcal{T}\mathrm{op}}(Y_{M^{-1}\mathbb{Z}}, Z)} \&\& {\operatorname{Map}_{\mathcal{T}\mathrm{op}}(X_{M^{-1}\mathbb{Z}}, Z)} \\
	\\
	{\operatorname{Map}_{\mathcal{T}\mathrm{op}}(Y, Z)} \&\& {\operatorname{Map}_{\mathcal{T}\mathrm{op}}(X, Z)}
	\arrow[from=1-1, to=1-3]
	\arrow[from=1-1, to=3-1]
	\arrow[from=1-3, to=3-3]
	\arrow[from=3-1, to=3-3]
\end{tikzcd}\]
whose vertical maps are weak equivalences by Lemma \ref{ineedthisonetoo}. It remains to show that (3) and (4) are equivalent. If (3) holds, then Lemma \ref{citeadnauseam}(2) implies that, for every space $Z$, the map $\operatorname{Map}_{\mathcal{T}\mathrm{op}}(f_{M^{-1}\mathbb{Z}}, Z)$ is a weak equivalence. Conversely, if (4) holds, then using the bijection of path components $[f_{M^{-1}\mathbb{Z}}, Z] \colon [Y_{M^{-1}\mathbb{Z}}, Z] \xrightarrow{\cong} [X_{M^{-1}\mathbb{Z}}, Z]$ for $Z=X_{M^{-1}\mathbb{Z}}$ and $Z=Y_{M^{-1}\mathbb{Z}}$ yields a homotopy inverse of $f_{M^{-1}\mathbb{Z}}$.
\end{proof}
It is the equivalence of (1) and (5) that is relevant to our work in Section \ref{section7}, namely its pointed version.

\begin{proposition}\label{illdothisonetoopt} Let $M$ be a multiplicative subset of $\mathbb{Z}$. A pointed map of $1$-connected CW complexes $f \colon X \rightarrow Y$ is an $M$-local homotopy equivalence if and only if, for every $M$-local pointed space $Z$, the map $\operatorname{Map}_{\mathcal{T}\mathrm{op}_{\ast}}(f, Z) \colon \operatorname{Map}_{\mathcal{T}\mathrm{op}_{\ast}}(Y, Z) \rightarrow \operatorname{Map}_{\mathcal{T}\mathrm{op}_{\ast}}(X, Z)$ is a weak equivalence.
\end{proposition}

\section{The family of model structures on $1$-reduced simplicial sets}\label{section7}
Let $M$ be a multiplicative subset of $\mathbb{Z}$. We left Bousfield localize the model category ${\mathcal{SS}\mathrm{ets}}_{1}$ to get a model category $\mathcal{L}_{M^{-1}\mathbb{Z}}{\mathcal{SS}\mathrm{ets}}_{1}$ in which the weak equivalences are the $M$-local homotopy equivalences and the fibrant objects are the $M$-local Kan complexes. In particular, we obtain a model category $\mathcal{L}_{\mathbb{Q}}{\mathcal{SS}\mathrm{ets}}_{1}$ whose weak equivalences are the rational homotopy equivalences and whose fibrant objects are the rational Kan complexes. Thus, we recover the family of model structures on $1$-reduced simplicial sets from \cite[Part II, Theorem 2.2]{Quillen1969} together with the characterization of fibrant objects in \cite[Part II, Corollary 2.6]{Quillen1969}. These model structures inherit the generating cofibrations and the simplicial structure that we found for ${\mathcal{SS}\mathrm{ets}}_{1}$.

\subsection{Left Bousfield localization}\label{section7.1} We first provide an overview of left Bousfield localization; a reference is \cite[Chapters 1-4]{Hirschhorn2003}.

Let $\mathcal{M}$ be a simplicial model category; in particular, for two objects $X$ and $Y$, we have a function complex $\operatorname{Map}_{\mathcal{M}}(X, Y)$. Though one can left Bousfield localize model categories that are not simplicial using homotopy function complexes \cite[Chapter 17]{Hirschhorn2003}, we only localize simplicial model categories in this paper, and our arguments are contingent upon our explicit understanding of the function complexes in our model categories.

Let $T$ be a set of morphisms in $\mathcal{M}$; in practice, it is important that $T$ is not a proper class.

\begin{definition}\label{deflblsdef}

\begin{enumerate}

\item An object $K$ of $\mathcal{M}$ is \emph{$T$-local} if $K$ is fibrant in $\mathcal{M}$ and, for every morphism $\lambda \colon A \rightarrow B$ in $T$, the map $\operatorname{Map}_{\mathcal{M}}(\lambda, K) \colon \operatorname{Map}_{\mathcal{M}}(B, K) \rightarrow \operatorname{Map}_{\mathcal{M}}(A, K)$ is a weak equivalence.

\item A morphism $f \colon X \rightarrow Y$ in $\mathcal{M}$ is a \emph{$T$-local equivalence} if, for every $T$-local object $K$, the map $\operatorname{Map}_{\mathcal{M}}(f, K) \colon \operatorname{Map}_{\mathcal{M}}(Y, K) \rightarrow \operatorname{Map}_{\mathcal{M}}(X, K)$ is a weak equivalence.

\item A \emph{$T$-localization} of an object $X$ of $\mathcal{M}$ is a $T$-local equivalence with $T$-local target $f \colon X \rightarrow X_{T}$.

\end{enumerate}

\end{definition}

\begin{remark}\label{deflblsdef1} Every morphism in $T$ and every weak equivalence in $\mathcal{M}$ is a $T$-local equivalence. Conversely, every $T$-local equivalence of $T$-local objects is a weak equivalence in $\mathcal{M}$ {\cite[Theorem 3.2.13(1)]{Hirschhorn2003}}.
\end{remark}
In light of our remark, the left Bousfield localization of $\mathcal{M}$ at $T$ enlarges the class of weak equivalences to incorporate the maps in $T$ while preserving the class of cofibrations. Consequently, the class of fibrations, and thus the class of fibrant objects, is forced to shrink. Though the fibrations in a left Bousfield localization can be quite difficult to describe, at least the fibrant objects are always known to us.

\begin{theorem}[{\cite[Theorem 4.7]{Barwick2010}}]\label{lblexists} Let $\mathcal{M}$ be a combinatorial, left proper, and simplicial model category. Let $T$ be a set of morphisms in $\mathcal{M}$. There is a model structure $\mathcal{L}_{T}\mathcal{M}$ on the underlying category of $\mathcal{M}$, called the \emph{left Bousfield localization of $\mathcal{M}$ at $T$}, in which a morphism $f$ is

\begin{enumerate}

\item a weak equivalence if and only if $f$ is a $T$-local equivalence; and

\item a cofibration if and only if $f$ is a cofibration in $\mathcal{M}$.

\end{enumerate}
An object $K$ is fibrant in $\mathcal{L}_{T}\mathcal{M}$ if and only if $K$ is $T$-local. The model category $\mathcal{L}_{T}\mathcal{M}$ is combinatorial and left proper, and it inherits the simplicial structure of $\mathcal{M}$.
\end{theorem}

\begin{remark}\label{lblexistsrmk} Note that a $T$-localization in $\mathcal{M}$ is a fibrant replacement in $\mathcal{L}_{T}\mathcal{M}$.
\end{remark}

As the class of cofibrations is preserved, $\mathcal{L}_{T}\mathcal{M}$ inherits the generating cofibrations of the combinatorial model category $\mathcal{M}$. However, the same is not true for the generating acyclic cofibrations because the class of weak equivalences is expanded and, in general, it is difficult to describe the smaller class of fibrations in a left Bousfield localization. Knowing the fibrant objects in a left Bousfield localization makes it feasible to write down a set of maps that detects them by tweaking any set that detects fibrant objects in $\mathcal{M}$; however, this set can get cumbersome to write out, which defeats its practical purpose. In our left Bousfield localizations in this paper, we have no nice description of such a set that detects fibrant objects, so we avoid writing one out altogether.

\subsection{The set for the left Bousfield localization}\label{section7.2} For a multiplicative subset $M$ of $\mathbb{Z}$, our characterization of $M$-local spaces from Proposition \ref{illdothisonept} suggests that we produce a $1$-reduced simplicial analog of the set of maps \[\{S^{n} \rightarrow S^{n}_{M^{-1}\mathbb{Z}} \mid n \geq 2\}\] in order to left Bousfield localize the model category ${\mathcal{SS}\mathrm{ets}}_{1}$ at that set. A first suggestion is to start with the simplicial $n$-sphere $S_{\mathrm{simp}}^{n}$ and use it to emulate the construction of $S^{n}_{M^{-1}\mathbb{Z}}$ in the $1$-reduced setting. This approach runs into a problem: the construction of $S^{n}_{M^{-1}\mathbb{Z}}$ involves self-maps of $S^{n}$ of arbitrary degree, whereas, by the following lemma, the simplicial $n$-sphere $S_{\mathrm{simp}}^{n}$ only has the trivial self-maps.

\begin{lemma}\label{oopsnoselfmaps} For $n \geq 2$, the only self-maps of $S_{\mathrm{simp}}^{n}$ are the constant map and the identity map.
\end{lemma}

\begin{proof} Let $f$ be a self-map of $S_{\mathrm{simp}}^{n}$. Let $x$ denote the $0$-simplex of $S_{\mathrm{simp}}^{n}$, and let $\epsilon$ denote the non-degenerate $n$-simplex of $S_{\mathrm{simp}}^{n}$. If $f_{n}(\epsilon)=\epsilon$, then $f$ is the identity map of $S_{\mathrm{simp}}^{n}$. If $f_{n}(\epsilon)$ is the iterated degeneracy of $x$, then $f$ is the constant self-map of $S_{\mathrm{simp}}^{n}$.
\end{proof}
Thus, we have to settle for an indirect route. We start with the CW inclusion $S^{n} \rightarrow S^{n}_{M^{-1}\mathbb{Z}}$, and we apply the singular complex functor to obtain the inclusion of simplicial sets $\operatorname{Sing}(S^{n}) \rightarrow \operatorname{Sing}(S^{n}_{M^{-1}\mathbb{Z}})$. Then, we take $1$st Eilenberg subcomplexes to get the inclusion of $1$-reduced simplicial sets \[E_{1}\operatorname{Sing}(S^{n}) \rightarrow E_{1}\operatorname{Sing}(S^{n}_{M^{-1}\mathbb{Z}}).\] To confirm that this map is our desired analog, we study the commutative diagram of CW complexes
\[\begin{tikzcd}[row sep=3mm, column sep=3mm, ampersand replacement=\&]
	{|E_{1}\operatorname{Sing}(S^{n})|} \&\& {|\operatorname{Sing}(S^{n})|} \&\& {S^{n}} \\
	\\
	{|E_{1}\operatorname{Sing}(S^{n}_{M^{-1}\mathbb{Z}})|} \&\& {|\operatorname{Sing}(S^{n}_{M^{-1}\mathbb{Z}})|} \&\& {S^{n}_{M^{-1}\mathbb{Z}},}
	\arrow[from=1-1, to=1-3]
	\arrow[from=1-1, to=3-1]
	\arrow[from=1-3, to=1-5]
	\arrow[from=1-3, to=3-3]
	\arrow[from=1-5, to=3-5]
	\arrow[from=3-1, to=3-3]
	\arrow[from=3-3, to=3-5]
\end{tikzcd}\]
where the right-hand horizontal maps arise from the geometric realization-singular complex adjunction.

\begin{proposition}\label{comdiagramheqs} Let $M$ be a multiplicative subset of $\mathbb{Z}$. Let $n \geq 2$. In the commutative diagram
\[\begin{tikzcd}[row sep=3mm, column sep=3mm, ampersand replacement=\&]
	{|E_{1}\operatorname{Sing}(S^{n})|} \&\& {|\operatorname{Sing}(S^{n})|} \&\& {S^{n}} \\
	\\
	{|E_{1}\operatorname{Sing}(S^{n}_{M^{-1}\mathbb{Z}})|} \&\& {|\operatorname{Sing}(S^{n}_{M^{-1}\mathbb{Z}})|} \&\& {S^{n}_{M^{-1}\mathbb{Z}},}
	\arrow[from=1-1, to=1-3]
	\arrow[from=1-1, to=3-1]
	\arrow[from=1-3, to=1-5]
	\arrow[from=1-3, to=3-3]
	\arrow[from=1-5, to=3-5]
	\arrow[from=3-1, to=3-3]
	\arrow[from=3-3, to=3-5]
\end{tikzcd}\]
all the horizontal maps are weak homotopy equivalences.
\end{proposition}

\begin{proof} By Corollary \ref{mayof67cor}, the horizontal maps in the left-hand square are weak homotopy equivalences. By Theorem \ref{quilleneqsecretly}(2), the horizontal maps in the right-hand square are weak homotopy equivalences, as well.
\end{proof}

\subsection{Description of the left Bousfield localization}\label{section7.3} For a multiplicative subset $M$ of $\mathbb{Z}$, we describe the left Bousfield localization of the model category ${\mathcal{SS}\mathrm{ets}}_{1}$ at the set of maps \[B_{M^{-1}\mathbb{Z}}=\{E_{1}\operatorname{Sing}(S^{n}) \rightarrow E_{1}\operatorname{Sing}(S^{n}_{M^{-1}\mathbb{Z}}) \mid n \geq 2\}.\] In particular, we characterize the $B_{M^{-1}\mathbb{Z}}$-local objects and the $B_{M^{-1}\mathbb{Z}}$-local equivalences in ${\mathcal{SS}\mathrm{ets}}_{1}$. To state our characterizations, we produce analogs of the notions from Section \ref{section6.1} for $1$-reduced simplicial sets.

\begin{definition}\label{defrhtssetsdef} Let $M$ be a multiplicative subset of $\mathbb{Z}$.

\begin{enumerate}

\item A map of $1$-reduced simplicial sets $f$ is an \emph{$M$-local homotopy equivalence} if $|f|$ is an $M$-local homotopy equivalence of $1$-connected spaces.

\item A $1$-reduced simplicial set $K$ is \emph{$M$-local} if $|K|$ is an $M$-local space.

\item An \emph{$M$-localization} of a $1$-reduced simplicial set $X$ is an $M$-local homotopy equivalence with $M$-local target $f \colon X \rightarrow X_{M^{-1}\mathbb{Z}}$.

\end{enumerate}
When $M^{-1}\mathbb{Z}=\mathbb{Q}$, we speak of a \emph{rational homotopy equivalence}, a \emph{rational simplicial set}, and a \emph{rationalization}.
\end{definition}
The following proposition helps us transition from $1$-reduced simplicial sets to $1$-connected spaces.

\begin{proposition}[{\cite[Proposition 1.1.11]{Hirschhorn2003}}]\label{pshsaveme} If $A$ and $X$ are pointed simplicial sets and $X$ is a Kan complex, then there is a natural weak equivalence $\operatorname{Map}_{\mathcal{SS}\mathrm{ets}_{\ast}}(A, X) \rightarrow \operatorname{Map}_{\mathcal{T}\mathrm{op}_{\ast}}(|A|, |X|)$.
\end{proposition}
We show that a $1$-reduced simplicial set is $B_{M^{-1}\mathbb{Z}}$-local if and only if it is an $M$-local Kan complex.

\begin{proposition}\label{charthelocals} Let $M$ be a multiplicative subset of $\mathbb{Z}$. Let $X$ be a $1$-reduced Kan complex. Let $n \geq 2$. The following statements are equivalent.

\begin{enumerate}

\item The map $\operatorname{Map}_{\mathcal{SS}\mathrm{ets}_{\ast}}(E_{1}\operatorname{Sing}(S^{n}_{M^{-1}\mathbb{Z}}), X) \rightarrow \operatorname{Map}_{\mathcal{SS}\mathrm{ets}_{\ast}}(E_{1}\operatorname{Sing}(S^{n}), X)$ is a weak equivalence.

\item The map $\operatorname{Map}_{\mathcal{SS}\mathrm{ets}_{\ast}}(\operatorname{Sing}(S^{n}_{M^{-1}\mathbb{Z}}), X) \rightarrow \operatorname{Map}_{\mathcal{SS}\mathrm{ets}_{\ast}}(\operatorname{Sing}(S^{n}), X)$ is a weak equivalence.

\item The map $\operatorname{Map}_{\mathcal{T}\mathrm{op}_{\ast}}(|\operatorname{Sing}(S^{n}_{M^{-1}\mathbb{Z}})|, |X|) \rightarrow \operatorname{Map}_{\mathcal{T}\mathrm{op}_{\ast}}(|\operatorname{Sing}(S^{n})|, |X|)$ is a weak equivalence.

\item The map $\operatorname{Map}_{\mathcal{T}\mathrm{op}_{\ast}}(S^{n}_{M^{-1}\mathbb{Z}}, |X|) \rightarrow \operatorname{Map}_{\mathcal{T}\mathrm{op}_{\ast}}(S^{n}, |X|)$ is a weak equivalence.

\item The space $|X|$ is $M$-local.

\end{enumerate}

\end{proposition}

\begin{proof} First, (1) and (2) are equivalent by Proposition \ref{comdiagramheqs}. Then, (2) and (3) are equivalent by Proposition \ref{pshsaveme}. Now, (3) and (4) are equivalent by Proposition \ref{comdiagramheqs}. Lastly, (4) and (5) are equivalent by Proposition \ref{illdothisonept}.
\end{proof}
By a similar series of characterizations, we establish that a map of $1$-reduced simplicial sets is a $B_{M^{-1}\mathbb{Z}}$-local equivalence if and only if it is an $M$-local homotopy equivalence. To that end, we need an auxiliary lemma.

\begin{lemma}\label{omgwhatshappening} Let $M$ be a multiplicative subset of $\mathbb{Z}$. Let $f \colon X \rightarrow Y$ be a map of $1$-reduced simplicial sets. The following statements are equivalent.

\begin{enumerate}

\item For every $M$-local Kan complex $K$, the map $\operatorname{Map}_{\mathcal{T}\mathrm{op}_{\ast}}(|f|, |K|)$ is a weak equivalence.

\item For every $M$-local pointed space $Z$, the map $\operatorname{Map}_{\mathcal{T}\mathrm{op}_{\ast}}(|f|, Z)$ is a weak equivalence.

\end{enumerate}

\end{lemma}

\begin{proof} Suppose (1) holds. Let $Z$ be an $M$-local pointed space. By Corollary \ref{forlaterusetbacor}, $E_{1}\operatorname{Sing}(Z)$ is a Kan complex. Corollary \ref{mayof67cor} and Theorem \ref{quilleneqsecretly}(2) provide the weak homotopy equivalences $|E_{1}\operatorname{Sing}(Z)| \rightarrow |\operatorname{Sing}(Z)| \rightarrow Z$, so $E_{1}\operatorname{Sing}(Z)$ is an $M$-local Kan complex. Now, there is a commutative diagram

\[\begin{tikzcd}[row sep=3mm, column sep=3mm, ampersand replacement=\&]
	{\operatorname{Map}_{\mathcal{T}\mathrm{op}_{\ast}}(|Y|, |E_{1}\operatorname{Sing}(Z)|)} \&\& {\operatorname{Map}_{\mathcal{T}\mathrm{op}_{\ast}}(|Y|, |\operatorname{Sing}(Z)|)} \&\& {\operatorname{Map}_{\mathcal{T}\mathrm{op}_{\ast}}(|Y|, Z)} \\
	\\
	{\operatorname{Map}_{\mathcal{T}\mathrm{op}_{\ast}}(|X|, |E_{1}\operatorname{Sing}(Z)|)} \&\& {\operatorname{Map}_{\mathcal{T}\mathrm{op}_{\ast}}(|X|, |\operatorname{Sing}(Z)|)} \&\& {\operatorname{Map}_{\mathcal{T}\mathrm{op}_{\ast}}(|X|, Z)}
	\arrow[from=1-1, to=1-3]
	\arrow[from=1-1, to=3-1]
	\arrow[from=1-3, to=1-5]
	\arrow[from=1-3, to=3-3]
	\arrow[from=1-5, to=3-5]
	\arrow[from=3-1, to=3-3]
	\arrow[from=3-3, to=3-5]
\end{tikzcd}\]
whose horizontal maps are weak equivalences by Corollary \ref{mayof67cor} and Theorem \ref{quilleneqsecretly}(2) and whose left-hand vertical map is a weak equivalence by (1); thus, so is the right-hand vertical map, proving (2). Conversely, suppose (2) holds. Let $K$ be an $M$-local simplicial set. Then, $|K|$ is an $M$-local pointed space, which implies (1).
\end{proof}
Using our lemma, we show that a $B_{M^{-1}\mathbb{Z}}$-local equivalence is precisely an $M$-local homotopy equivalence.

\begin{proposition}\label{chartheloceqs} Let $M$ be a multiplicative subset of $\mathbb{Z}$. Let $f$ be a map of $1$-reduced simplicial sets. The following statements are equivalent.

\begin{enumerate}

\item For every $B_{M^{-1}\mathbb{Z}}$-local simplicial set $K$, the map $\operatorname{Map}_{\mathcal{SS}\mathrm{ets}_{\ast}}(f, K)$ is a weak equivalence.

\item For every $M$-local Kan complex $K$, the map $\operatorname{Map}_{\mathcal{SS}\mathrm{ets}_{\ast}}(f, K)$ is a weak equivalence.

\item For every $M$-local Kan complex $K$, the map $\operatorname{Map}_{\mathcal{T}\mathrm{op}_{\ast}}(|f|, |K|)$ is a weak equivalence.

\item For every $M$-local pointed space $Z$, the map $\operatorname{Map}_{\mathcal{T}\mathrm{op}_{\ast}}(|f|, Z)$ is a weak equivalence.

\item The map $|f|$ is an $M$-local homotopy equivalence.

\end{enumerate}

\end{proposition}

\begin{proof} First, (1) and (2) are equivalent by Proposition \ref{charthelocals}. Then, (2) and (3) are equivalent by Proposition \ref{pshsaveme}. Now, (3) and (4) are equivalent by Lemma \ref{omgwhatshappening}. Lastly, (4) and (5) are equivalent by Proposition \ref{illdothisonetoopt}.
\end{proof}
Propositions \ref{charthelocals} and \ref{chartheloceqs} together describe the left Bousfield localization of ${\mathcal{SS}\mathrm{ets}}_{1}$ at the set $B_{M^{-1}\mathbb{Z}}$.

\begin{theorem}\label{ididitididit} Let $M$ be a multiplicative subset of $\mathbb{Z}$. There is a model structure $\mathcal{L}_{M^{-1}\mathbb{Z}}{\mathcal{SS}\mathrm{ets}}_{1}$ on the category of $1$-reduced simplicial sets in which a map $f$ is

\begin{enumerate}

\item a weak equivalence if and only if $f$ is an $M$-local homotopy equivalence; and

\item a cofibration if and only if $f$ is a monomorphism.

\end{enumerate}
A $1$-reduced simplicial set $X$ is fibrant in $\mathcal{L}_{M^{-1}\mathbb{Z}}{\mathcal{SS}\mathrm{ets}}_{1}$ if and only if $X$ is an $M$-local Kan complex. The model category $\mathcal{L}_{M^{-1}\mathbb{Z}}{\mathcal{SS}\mathrm{ets}}_{1}$ is combinatorial, left proper, and simplicial with the structure from Theorem \ref{simplicialdefsworks}.
\end{theorem}
Theorem \ref{ididitididit} recovers the family of model structures on $1$-reduced simplicial sets in \cite[Part II, Theorem 2.2]{Quillen1969}, as well as the characterization of fibrant objects in \cite[Part II, Corollary 2.6]{Quillen1969}. The most prominent model structure in this family is that corresponding to the case $M^{-1}\mathbb{Z}=\mathbb{Q}$, which we present separately.

\begin{corollary}\label{ididitididitq} There is a model structure $\mathcal{L}_{\mathbb{Q}}{\mathcal{SS}\mathrm{ets}}_{1}$ on the category of $1$-reduced simplicial sets in which

\begin{enumerate}

\item a map is a weak equivalence if and only if it is a rational homotopy equivalence; and

\item a $1$-reduced simplicial set is fibrant if and only if it is a rational Kan complex.

\end{enumerate}

\end{corollary}

Lastly, note that our approach establishes some additional model-categorical features of the model categories in Quillen's family; namely, we know their simplicial structure, as well as a generating set of cofibrations.

\section{An approach using localization with respect to homology}\label{section8}
We give an alternative description of the family of model structures from Theorem \ref{ididitididit}. This description uses Bousfield localization with respect to homology, as developed by Bousfield in \cite{Bousfield1975}, together with Theorem \ref{leftinducetheoremwit} on left-induced model structures. The idea is that, instead of first left-transferring the model structure on ${\mathcal{SS}\mathrm{ets}}_{\ast}$ and then localizing, we can first localize the model structure on ${\mathcal{SS}\mathrm{ets}}_{\ast}$ and then do a left transfer. We localize ${\mathcal{SS}\mathrm{ets}}_{\ast}$ with respect to homology. Because homotopy and homology are in agreement for $1$-connected spaces, the left transfer to ${\mathcal{SS}\mathrm{ets}}_{1}$ yields the family of model structures from Theorem \ref{ididitididit}.

\subsection{Bousfield localization with respect to homology}\label{section8.1} Firstly, we introduce the following terminology.

\begin{definition}\label{defhlgyisodef} Let $M$ be a multiplicative subset of $\mathbb{Z}$. A map of topological spaces $f \colon X \rightarrow Y$ is an \emph{$M$-local homology isomorphism} if, for every $n \geq 0$, the map $H_{n}(f; M^{-1}\mathbb{Z}) \colon H_{n}(X; M^{-1}\mathbb{Z}) \rightarrow H_{n}(Y; M^{-1}\mathbb{Z})$ is an isomorphism. A simplicial map $g$ is an \emph{$M$-local homology isomorphism} if $|g|$ is an $M$-local homology isomorphism of topological spaces. When $M^{-1}\mathbb{Z}=\mathbb{Q}$, we speak of a \emph{rational homology isomorphism}.
\end{definition}

\begin{remark}\label{defhlgyisodef1} As the abelian group $M^{-1}\mathbb{Z}$ is flat \cite[Theorem 4.80]{Rotman2009}, the universal coefficient theorem for homology yields the natural isomorphism $H_{n}(X; M^{-1}\mathbb{Z}) \cong H_{n}(X; \mathbb{Z}) \otimes M^{-1}\mathbb{Z}$. Thus, $M$-local homology discards torsion from integral homology just as $M$-local homotopy discards torsion from the higher homotopy groups.
\end{remark}

The following example is invoked later in this paper.

\begin{example}\label{defhlgyisodef2} By Remark \ref{defhlgyisodef1}, the unique map from $\mathbb{R}P^{2}$ to the point is a rational homology isomorphism.
\end{example}
The natural isomorphism from Remark \ref{defhlgyisodef1} also has the following consequence that we note for later use.

\begin{lemma}\label{forlateruselemma} Let $M$ be a multiplicative subset of $\mathbb{Z}$. Every integral homology isomorphism is an $M$-local homology isomorphism.
\end{lemma}
Also for later use, we need to recall the following theorem that relates integral homology and homotopy.

\begin{theorem}[{\cite[Corollary 6.69(1)]{DavisKirk2001}}]\label{dkcitediguess} A weak homotopy equivalence is an integral homology isomorphism.
\end{theorem}
We also need to recall another theorem of a similar nature for future reference; this theorem informs us that, under the hypothesis of $1$-connectedness, local homology agrees with local homotopy.

\begin{theorem}[{\cite[Theorem 8.6]{FHT2001}}]\label{whiteheadserrepair} Let $M$ be a multiplicative subset of $\mathbb{Z}$. A map of $1$-connected spaces $f$ is an $M$-local homology isomorphism if and only if $f$ is an $M$-local homotopy equivalence.
\end{theorem}
For a multiplicative subset $M$ of $\mathbb{Z}$, Bousfield localization with respect to homology endows the category $\mathcal{SS}\mathrm{ets}$ of simplicial sets with a model structure whose weak equivalences are the $M$-local homology isomorphisms. We use the pointed version of this result, which holds by the same proof.

\begin{theorem}[{\cite[Theorem 10.2]{Bousfield1975}}]\label{ptedhomloc} Let $M$ be a multiplicative subset of $\mathbb{Z}$. There exists a model structure $\mathcal{L}_{HM^{-1}\mathbb{Z}}\mathcal{SS}\mathrm{ets}_{\ast}$ on the category of pointed simplicial sets in which a map $f$ is

\begin{enumerate}

\item a weak equivalence if and only if $f$ is an $M$-local homology isomorphism; and

\item a cofibration if and only if $f$ is a monomorphism.

\end{enumerate}
This model category is combinatorial, left proper, and simplicial with the structure from Example \ref{defsmcdef2}.
\end{theorem}

\begin{remark}\label{ptedhomlocrmk} By \cite[\S 2.5]{Bousfield1997}, $\mathcal{L}_{HM^{-1}\mathbb{Z}}\mathcal{SS}\mathrm{ets}_{\ast}$ is the left Bousfield localization of $\mathcal{SS}\mathrm{ets}_{\ast}$ at a single map, which allows us to infer that $\mathcal{L}_{HM^{-1}\mathbb{Z}}\mathcal{SS}\mathrm{ets}_{\ast}$ inherits the properties of being combinatorial, left proper, and simplicial.
\end{remark}
In particular, we have a model structure whose weak equivalences are the rational homology isomorphisms.

\subsection{An alternative description of the family of model structures on $1$-reduced simplicial sets}\label{section8.2}
For a multiplicative subset $M$ of $\mathbb{Z}$, we show that the model category $\mathcal{L}_{M^{-1}\mathbb{Z}}{\mathcal{SS}\mathrm{ets}}_{1}$ from Theorem \ref{ididitididit} is the left transfer of the model category $\mathcal{L}_{HM^{-1}\mathbb{Z}}\mathcal{SS}\mathrm{ets}_{\ast}$ from Theorem \ref{ptedhomloc} along the adjunction \[F \colon {\mathcal{SS}\mathrm{ets}}_{1} \rightleftarrows \mathcal{SS}\mathrm{ets}_{\ast} \colon E_{1}\] from Proposition \ref{yesquite}(2). Again, we use Theorem \ref{leftinducetheoremwit}, so we verify the homotopical condition in its hypothesis.

\begin{lemma}\label{verifybhkkrs2} Let $M$ be a multiplicative subset of $\mathbb{Z}$. Let $f$ be a map of $1$-reduced simplicial sets. If $f$ has the right lifting property with respect to every monomorphism of $1$-reduced simplicial sets, then $f$ is an $M$-local homology isomorphism.
\end{lemma}

\begin{proof} Suppose $f$ has the right lifting property with respect to every monomorphism of $1$-reduced simplicial sets. By Lemma \ref{verifybhkkrs}, we know that $|f|$ is a weak homotopy equivalence. Then, Theorem \ref{dkcitediguess} informs us that $|f|$ is an integral homology isomorphism. Therefore, by Lemma \ref{forlateruselemma}, $|f|$ is an $M$-local homology isomorphism.
\end{proof}
We use Theorem \ref{leftinducetheoremwit} to give an alternative approach to the family of model structures from Theorem \ref{ididitididit}.

\begin{theorem}\label{rightinduceeyebrow2} Let $M$ be a multiplicative subset of $\mathbb{Z}$. The model category $\mathcal{L}_{M^{-1}\mathbb{Z}}{\mathcal{SS}\mathrm{ets}}_{1}$ is the left transfer of the model category $\mathcal{L}_{HM^{-1}\mathbb{Z}}\mathcal{SS}\mathrm{ets}_{\ast}$ along the adjunction $F \colon {\mathcal{SS}\mathrm{ets}}_{1} \rightleftarrows {\mathcal{SS}\mathrm{ets}}_{\ast} \colon E_{1}$.
\end{theorem}

\begin{proof} Theorem \ref{leftinducetheoremwit} says that we can left transfer the model structure $\mathcal{L}_{HM^{-1}\mathbb{Z}}\mathcal{SS}\mathrm{ets}_{\ast}$ along the adjunction to a model structure on ${\mathcal{SS}\mathrm{ets}}_{1}$ whose weak equivalences are the $M$-local homology isomorphisms and whose cofibrations are the monomorphisms. By Theorem \ref{whiteheadserrepair}, a map of $1$-reduced simplicial sets is an $M$-local homology isomorphism if and only if it is an $M$-local homotopy equivalence, so the model structure is $\mathcal{L}_{M^{-1}\mathbb{Z}}{\mathcal{SS}\mathrm{ets}}_{1}$.
\end{proof}

In particular, we obtain another modern description of the model category $\mathcal{L}_{\mathbb{Q}}{\mathcal{SS}\mathrm{ets}}_{1}$ from Corollary \ref{ididitididitq}.

\section{Rational homotopy theory of non-simply connected spaces}\label{section9}
Let $M$ be a multiplicative subset of $\mathbb{Z}$. We study some elements of the rational homotopy theory of non-simply connected spaces in order to generalize Quillen's family of model structures beyond $1$-reduced simplicial sets. In particular, we retain our description of $M$-local spaces in terms of function complexes beyond the $1$-connected setting, which allows us to describe our left Bousfield localization of all spaces in Section \ref{section10}.

\subsection{Rationalization of spaces}\label{section9.1}
Firstly, Definition \ref{defrhtdef} generalizes to all spaces as follows.

\begin{definition}[{\cite[\S 6]{GTHT2000}}]\label{defrhtnscdef} Let $M$ be a multiplicative subset of $\mathbb{Z}$.

\begin{enumerate}

\item A map $f \colon X \rightarrow Y$ is an \emph{$M$-local homotopy equivalence} if $f_{\ast} \colon \pi_{n}(X) \rightarrow \pi_{n}(Y)$ is an isomorphism for every $n \leq 1$ and $f_{\ast} \otimes \operatorname{id}_{M^{-1}\mathbb{Z}} \colon \pi_{n}(X) \otimes M^{-1}\mathbb{Z} \rightarrow \pi_{n}(Y) \otimes M^{-1}\mathbb{Z}$ is an isomorphism for every $n \geq 2$.

\item A space $Y$ is \emph{$M$-local} if, for every $n \geq 2$, the abelian group $\pi_{n}(Y)$ is an $M^{-1}\mathbb{Z}$-module.

\item An \emph{$M$-localization} of a space $Y$ is an $M$-local homotopy equivalence with $M$-local target $f \colon Y \rightarrow Y_{M^{-1}\mathbb{Z}}$.

\end{enumerate}
When $M^{-1}\mathbb{Z}=\mathbb{Q}$, we speak of a \emph{rational homotopy equivalence}, a \emph{rational space}, and a \emph{rationalization}.

\end{definition}
The absence of any nilpotency hypothesis deprives us of any sensible notion of tensoring the fundamental group with $\mathbb{Q}$, so we ask that a rational homotopy equivalence induce an isomorphism of fundamental groups.

\begin{example}\label{defrhtnscdef1} Every weak homotopy equivalence is an $M$-local homotopy equivalence. Also, every $M$-local homotopy equivalence of $M$-local spaces is a weak homotopy equivalence.
\end{example}

The following example ought to be contrasted with Example \ref{defhlgyisodef2}.

\begin{example}\label{defrhtnscdef22} The unique map from $\mathbb{R}P^{2}$ to the point is not a rational homotopy equivalence; the nontriviality of $\pi_{1}(\mathbb{R}P^{2}) \cong \mathbb{Z}/2\mathbb{Z}$ precludes it from inducing an isomorphism of fundamental groups.
\end{example}

\begin{example}\label{defrhtnscdef2} The spaces $S^{1}$, $S^{1} \vee S^{1}$, and $\mathbb{R}P^{\infty}$ are $M$-local because their higher homotopy groups vanish.
\end{example}
Example \ref{defrhtnscdef2} begs the question of why our definition of an $M$-local space imposes no condition on the fundamental group. In particular, we can ask that, for every element $m$ of $M$, the $m$-th power map $g \mapsto g^{m}$ be a self-bijection of the fundamental group, even though it need not be a homomorphism for non-abelian groups. The problem is that, with this stricter definition, there are spaces that do not admit any $M$-localization; this result is proven by Casacuberta in \cite[Proposition 2.1]{Casacuberta1993}. In our generalization of rational homotopy theory, spaces can be rationalized, albeit not with the universal property of the rationalization of $1$-connected spaces.

\begin{theorem}[{\cite[\S 6]{GTHT2000}}]\label{thebigdealsad} Let $M$ be a multiplicative subset of $\mathbb{Z}$. Every space admits an $M$-localization.
\end{theorem}
The proof of the theorem sheds light on what goes wrong with the universal property of rationalization.

\begin{proof} If $X$ is a connected space, its $1$st Postnikov approximation yields the fibration $\widetilde{X} \rightarrow X \rightarrow K(\pi_{1}(X), 1)$, where $\widetilde{X}$ is the universal cover of $X$ and $K(\pi_{1}(X), 1)$ is an Eilenberg-Mac Lane space. Since $\widetilde{X}$ is a $1$-connected space, we invoke fiberwise $M$-localization, in the sense of \cite[Theorem 6.1.3]{Hirschhorn2003}, to obtain a map of fibrations
\[\begin{tikzcd}[row sep=3mm, column sep=3mm, ampersand replacement=\&]
	{\widetilde{X}} \&\& {\widetilde{X}_{M^{-1}\mathbb{Z}}} \\
	\\
	X \&\& {X_{M^{-1}\mathbb{Z}}} \\
	\\
	\& {K(\pi_{1}(X), 1),}
	\arrow[from=1-1, to=1-3]
	\arrow[from=1-1, to=3-1]
	\arrow[from=1-3, to=3-3]
	\arrow[from=3-1, to=3-3]
	\arrow[from=3-1, to=5-2]
	\arrow[from=3-3, to=5-2]
\end{tikzcd}\]
where the map on fibers $\widetilde{X} \rightarrow \widetilde{X}_{M^{-1}\mathbb{Z}}$ is an $M$-localization of the $1$-connected space $\widetilde{X}$ and, by the induced map of long exact sequences of homotopy groups, the map on total spaces $X \rightarrow X_{M^{-1}\mathbb{Z}}$ is the required $M$-localization of $X$. Applying this treatment to each path component yields the claim for disconnected spaces, as well.
\end{proof}
The downside with fiberwise localization in our setting is that, because it applies to fibrations over a fixed base space $B$, its universal property states that it is initial among fibrations over $B$ with local fibers. Thus, our application of fiberwise localization only yields a universal property of $M$-localization over $K(\pi_{1}(X), 1)$, and we are neither aware of how to get rid of this restriction nor convinced that we should even be able to do so.

The universal property of rationalization was pivotal for our characterization of rational homotopy equivalences in terms of function complexes, which gave us the description of the weak equivalences in the left Bousfield localization in Section \ref{section7}. Nevertheless, fiberwise localization still allows us to infer that, in our localization of all spaces in Section \ref{section10}, the weak equivalences are the $M$-local homotopy equivalences in our generalized sense.

\subsection{The construction of local spheres}\label{section9.2}
Let $M$ be a multiplicative subset of $\mathbb{Z}$ and let $n \geq 2$. To salvage our function complex characterization of $M$-local spaces in the absence of a universal property for $M$-localizatrion, we need the details of the CW complex construction of the $M$-local $n$-sphere $S^{n}_{M^{-1}\mathbb{Z}}$ from \cite[Chapter 9(a)]{FHT2001}.

We enumerate $M \cap \mathbb{N}=\{m_{1}, m_{2}, m_{3}, \dots\}.$ We denote the $(n+1)$-disk by $D^{n+1}$. Then, $S^{n}_{M^{-1}\mathbb{Z}}$ is the pushout
\[\begin{tikzcd}[row sep=3mm, column sep=3mm, ampersand replacement=\&]
	{\coprod\limits_{j=1}^{\infty} (S^{n})_{j}} \&\& {\bigvee\limits_{i=0}^{\infty} (S^{n})_{i}} \\
	\\
	{\coprod\limits_{j=1}^{\infty} (D^{n+1})_{j}} \&\& {S^{n}_{M^{-1}\mathbb{Z}},}
	\arrow["h", from=1-1, to=1-3]
	\arrow[from=1-1, to=3-1]
	\arrow[from=1-3, to=3-3]
	\arrow[from=3-1, to=3-3]
\end{tikzcd}\]
where, for $j \geq 1$, the $j$-th $(n+1)$-cell $(D^{n+1})_{j}$ is attached to $\bigvee\limits_{i=0}^{\infty} (S^{n})_{i}$ by an attaching map $h$ given by
\[S^{n} \xrightarrow{\textrm{pinch}} S^{n} \vee S^{n} \xrightarrow{\operatorname{id}_{S^{n}} \vee (-m_{j})} (S^{n})_{j-1} \vee (S^{n})_{j} \xrightarrow{\textrm{inclusion}} \bigvee\limits_{i=0}^{\infty} (S^{n})_{i}.\]
The map $(-m_{j})$ is a self-map of $S^{n}$ that has degree $(-m_{j})$. In other words, the $1$-connected CW complex \[S^{n}_{M^{-1}\mathbb{Z}}=\left(\bigvee\limits_{i=0}^{\infty} (S^{n})_{i}\right) \underset{h}{\bigcup} \left(\coprod\limits_{j=1}^{\infty} (D^{n+1})_{j}\right)\] is the mapping telescope of the sequence of maps $S^{n} \xrightarrow{m_{1}} S^{n} \xrightarrow{m_{2}} S^{n} \xrightarrow{m_{3}} \cdots$.

\subsection{A characterization of rational spaces}\label{section9.3}
We use the construction of $S^{n}_{M^{-1}\mathbb{Z}}$ to prove a key result.

\begin{proposition}\label{thebigdealnonsc} Let $M$ be a multiplicative subset of $\mathbb{Z}$ and let $n \geq 2$. Let $Y$ be a space such that $\pi_{n}(Y)$ is an $M^{-1}\mathbb{Z}$-module. Then, every map $f \colon S^{n} \rightarrow Y$ extends uniquely up to homotopy to a map $g \colon S^{n}_{M^{-1}\mathbb{Z}} \rightarrow Y$ such that the diagram

\[\begin{tikzcd}[row sep=3mm, column sep=3mm, ampersand replacement=\&]
	{S^{n}} \&\& {S^{n}_{M^{-1}\mathbb{Z}}} \\
	\\
	\&\& {Y}
	\arrow[from=1-1, to=1-3]
	\arrow["f"', from=1-1, to=3-3]
	\arrow["g"', dashed, from=1-3, to=3-3]
\end{tikzcd}\]
commutes.
\end{proposition}

\begin{proof} Given a map $f \colon S^{n} \rightarrow Y$, we know that $[f]$ is an element of the $M^{-1}\mathbb{Z}$-module $\pi_{n}(Y)$. Thus, there exists a unique element $[u]$ of $\pi_{n}(Y)$ such that $[f]=m_{1} \cdot [u]$, and the representative $u \colon S^{n} \rightarrow Y$ of $[u]$ is unique up to homotopy. Using the notation in our construction of $S^{n}_{M^{-1}\mathbb{Z}}$, we extend $f$ from $S^{n}=(S^{n})_{0}$ to $(S^{n})_{0} \vee (S^{n})_{1}$ uniquely up to homotopy by setting $g|_{(S^{n})_{1}}=u$. As $[f]-(m_{1} \cdot [u])=0 $, we can further extend $f$ to \[\left((S^{n})_{0} \vee (S^{n})_{1}\right) \underset{h}{\bigcup} (D^{n+1})_{1}\] uniquely up to homotopy. Continuing further down the telescope in this fashion yields the claim.
\end{proof}
Proposition \ref{thebigdealnonsc} salvages our characterization of $M$-local spaces in terms of function complexes.

\begin{corollary}\label{illdothisonenonsc} Let $M$ be a multiplicative subset of $\mathbb{Z}$. A space $Y$ is $M$-local if and only if, for every $n \geq 2$, the restriction map $\operatorname{Map}_{\mathcal{T}\mathrm{op}}(S^{n}_{M^{-1}\mathbb{Z}}, Y) \rightarrow \operatorname{Map}_{\mathcal{T}\mathrm{op}}(S^{n}, Y)$ is a weak equivalence.
\end{corollary}

\begin{proof} If $Y$ is $M$-local, then Proposition \ref{thebigdealnonsc} implies that the restriction map is a weak equivalence. The converse implication is established as in the proof of Proposition \ref{illdothisone}.
\end{proof}

\section{The family of model structures on spaces}\label{section10}
Let $M$ be a multiplicative subset of $\mathbb{Z}$. We left Bousfield localize the model category ${\mathcal{SS}\mathrm{ets}}$ of simplicial sets to obtain a model category $\mathcal{L}_{M^{-1}\mathbb{Z}}{\mathcal{SS}\mathrm{ets}}$ in which the weak equivalences are the $M$-local homotopy equivalences in our extended sense and the fibrant objects are the $M$-local Kan complexes. In particular, we produce a model category $\mathcal{L}_{\mathbb{Q}}{\mathcal{SS}\mathrm{ets}}$ whose weak equivalences are the rational homotopy equivalences and whose fibrant objects are the rational Kan complexes. Thus, we generalize Quillen's family of model structures to all spaces. Lastly, we show that our localization of spaces at rational homotopy does not coincide with Bousfield's localization at rational homology as rational homotopy and rational homology do not agree beyond the $1$-connected setting.

\subsection{The family of model structures on simplicial sets}\label{section10.1}
Proposition \ref{thebigdealnonsc} allows us to describe the left Bousfield localization of ${\mathcal{SS}\mathrm{ets}}$ at the set of maps \[T_{M^{-1}\mathbb{Z}}=\{\operatorname{Sing}(S^{n}) \rightarrow \operatorname{Sing}(S^{n}_{M^{-1}\mathbb{Z}}) \mid n \geq 2\}.\]

To do so, we first formulate the analog of Definition \ref{defrhtnscdef} for simplicial sets.

\begin{definition}\label{defrhtssetsnscdef} Let $M$ be a multiplicative subset of $\mathbb{Z}$. 

\begin{enumerate}

\item A simplicial map $f$ is an \emph{$M$-local homotopy equivalence} if $|f|$ is an $M$-local homotopy equivalence.

\item A simplicial set $K$ is \emph{$M$-local} if $|K|$ is an $M$-local space.

\item An \emph{$M$-localization} of $X$ is an $M$-local homotopy equivalence with $M$-local target $X \rightarrow X_{M^{-1}\mathbb{Z}}$.

\end{enumerate}
When $M^{-1}\mathbb{Z}=\mathbb{Q}$, we speak of a \emph{rational homotopy equivalence}, a \emph{rational simplicial set}, and a \emph{rationalization}.
\end{definition}

Having salvaged our characterization of $M$-local spaces in terms of function complexes, we can replicate our proof of Proposition \ref{charthelocals} to show that a simplicial set is $T_{M^{-1}\mathbb{Z}}$-local if and only if it is an $M$-local Kan complex. It remains to show that a $T_{M^{-1}\mathbb{Z}}$-local equivalence is precisely an $M$-local homotopy equivalence. To that end, we invoke a general result on fiberwise localizations for the fiberwise $M$-localization from Proposition \ref{thebigdealsad}.

\begin{lemma}{\cite[Theorem 6.1.3(3)]{Hirschhorn2003}}\label{prefiblocweqs} Let $M$ be a multiplicative subset of $\mathbb{Z}$. Every fiberwise $M$-localization map is a $T_{M^{-1}\mathbb{Z}}$-local equivalence.
\end{lemma}

Lemma \ref{prefiblocweqs} allows us to characterize the $T_{M^{-1}\mathbb{Z}}$-local equivalences in our desired fashion.

\begin{proposition}\label{fiblocweqs} Let $M$ be a multiplicative subset of $\mathbb{Z}$. Let $f \colon X \rightarrow Y$ be a simplicial map. The following statements are equivalent.

\begin{enumerate}

\item The map $f$ is a $T_{M^{-1}\mathbb{Z}}$-local equivalence.

\item The induced map on fiberwise $M$-localizations $f_{M^{-1}\mathbb{Z}}$ is a $T_{M^{-1}\mathbb{Z}}$-local equivalence.

\item The induced map on fiberwise $M$-localizations $f_{M^{-1}\mathbb{Z}}$ is a weak equivalence.

\item The induced map on fiberwise $M$-localizations $f_{M^{-1}\mathbb{Z}}$ is an $M$-local homotopy equivalence.

\item The map $f$ is an $M$-local homotopy equivalence.

\end{enumerate}

\end{proposition}

\begin{proof} Fiberwise $M$-localization induces the commutative square

\[\begin{tikzcd}[row sep=3mm, column sep=3mm, ampersand replacement=\&]
	X \&\& Y \\
	\\
	{X_{M^{-1}\mathbb{Z}}} \&\& {Y_{M^{-1}\mathbb{Z}}.}
	\arrow[from=1-1, to=1-3]
	\arrow[from=1-1, to=3-1]
	\arrow[from=1-3, to=3-3]
	\arrow[from=3-1, to=3-3]
\end{tikzcd}\]

Firstly, (1) and (2) are equivalent because the vertical maps are $T_{M^{-1}\mathbb{Z}}$-local equivalences by Lemma \ref{prefiblocweqs}. Then, (2) and (3) are equivalent by Remark \ref{deflblsdef1}. Now, (3) and (4) are equivalent by Example \ref{defrhtnscdef1}. Lastly, (4) and (5) are equivalent because the vertical maps are $M$-local homotopy equivalences by Proposition \ref{thebigdealsad}.

\end{proof}

We arrive at our promised generalization of Quillen's family of model structures to all spaces.

\begin{theorem}\label{ididitididitnsc} Let $M$ be a multiplicative subset of $\mathbb{Z}$. There is a model structure $\mathcal{L}_{M^{-1}\mathbb{Z}}{\mathcal{SS}\mathrm{ets}}$ on the category of simplicial sets in which a map $f$ is

\begin{enumerate}

\item a weak equivalence if and only if $f$ is an $M$-local homotopy equivalence; and

\item a cofibration if and only if $f$ is a monomorphism.

\end{enumerate}

A simplicial set $X$ is fibrant in $\mathcal{L}_{M^{-1}\mathbb{Z}}{\mathcal{SS}\mathrm{ets}}$ if and only if $X$ is an $M$-local Kan complex. The model category $\mathcal{L}_{M^{-1}\mathbb{Z}}{\mathcal{SS}\mathrm{ets}}$ is combinatorial, left proper, and simplicial with the structure from Example \ref{defsmcdef1}.
\end{theorem}
Again, a distinguished model structure in this family is that corresponding to the rational case $M^{-1}\mathbb{Z}=\mathbb{Q}$.

\begin{corollary}\label{ididitididitqnsc} There is a model structure $\mathcal{L}_{\mathbb{Q}}{\mathcal{SS}\mathrm{ets}}$ on the category of simplicial sets in which

\begin{enumerate}

\item a map is a weak equivalence if and only if it is a rational homotopy equivalence; and

\item a simplicial set is fibrant if and only if it is a rational Kan complex.

\end{enumerate}

\end{corollary}

Thus, our modern approach to Quillen's model structures on $1$-reduced simplicial sets generalizes to all spaces and yields a model category that encodes the rational homotopy theory of non-simply connected spaces.

\subsection{The model category of topological spaces}\label{section10.2}
Quillen's family of model structures on the category of $1$-reduced simplicial sets finds no counterpart for the category of $1$-connected spaces, for the latter category does not have all finite limits and colimits. For instance, consider the pullback
\[\begin{tikzcd}[row sep=3mm, column sep=3mm, ampersand replacement=\&]
	{S^{1}} \&\& {\mathbb{R}^{2}} \\
	\\
	{|\Delta[0]|} \&\& {\mathbb{R},}
	\arrow[from=1-1, to=1-3]
	\arrow[from=1-1, to=3-1]
	\arrow["f"', from=1-3, to=3-3]
	\arrow["1", from=3-1, to=3-3]
\end{tikzcd}\]
where $f$ is given by the formula $f(x, y)=x^{2}+y^{2}$, so $f^{-1}(1)=S^{1}$ is the unit circle in the plane. In this pullback, we start with the contractible spaces $|\Delta[0]|$, $\mathbb{R}^{2}$, and $\mathbb{R}$ and get $S^{1}$, which has fundamental group $\pi_{1}(S^{1}) \cong \mathbb{Z}$. However, now that we have a family of model structures on simplicial sets, we can produce its analog for the category $\mathcal{T}\mathrm{op}$ of topological spaces by left Bousfield localizing the model structure on $\mathcal{T}\mathrm{op}$ at our usual set.

We first recall the model structure on $\mathcal{T}\mathrm{op}$ from \cite[\S 8]{DwyerSpalinski1995}; the description is in terms of the following definition.

\begin{definition}\label{defauxfortopdef}

\begin{enumerate}

\item Given a sequence of inclusions $X_{0} \rightarrow X_{1} \rightarrow \cdots$ such that the pairs $(X_{n+1}, X_{n})$ are relative CW complexes, we say that $X_{0} \rightarrow \operatorname{colim}(X_{n})$ is a \emph{generalized relative CW inclusion}.

\item A \emph{Serre fibration} is a map that has the right lifting property with respect to the class of inclusions $\{A \times \{0\} \rightarrow A \times [0, 1] \mid A \textrm{ is a CW complex}\}$.

\end{enumerate}

\end{definition}

In other words, a Serre fibration is a map with the homotopy lifting property against all CW complexes.

\begin{theorem}[{\cite[Propositions 8.3 and 8.9]{DwyerSpalinski1995}}]\label{modcattopsp} The category ${\mathcal{T}\mathrm{op}}$ of topological spaces has a model structure in which a map $f$ is

\begin{enumerate}

\item a weak equivalence if and only if $f$ is a weak homotopy equivalence;

\item a cofibration if and only if $f$ is a retract of a generalized relative CW inclusion; and

\item a fibration if and only if $f$ is a Serre fibration.

\end{enumerate}

\end{theorem}

The model category of spaces is also cofibrantly generated and simplicial with the following structure.

\begin{proposition}[{\cite[Example 11.1.8]{Hirschhorn2003}}]\label{topcofgenv1} The model category ${\mathcal{T}\mathrm{op}}$ is cofibrantly generated.

\begin{enumerate}

\item A generating set of cofibrations is $|I|=\{|i_{n}| \colon |\partial\Delta[n]| \rightarrow |\Delta[n]| \mid n \geq 0\}$.

\item A generating set of acyclic cofibrations is $|J|=\{|j_{n, k}| \colon |\Lambda[n, k]| \rightarrow |\Delta[n]| \mid n \geq 1 \textrm{ and } 0 \leq k \leq n\}$.

\end{enumerate}

\end{proposition}

\begin{proposition}[{\cite[Example 9.1.15]{Hirschhorn2003}}]\label{topsimpv1} Given two spaces $X$ and $Y$ and a simplicial set $Z$, we define

\begin{enumerate}

\item the function complex $\operatorname{Map}_{{\mathcal{T}\mathrm{op}}}(X, Y)$ to be the simplicial set with $n$-simplices $\operatorname{Hom}_{{\mathcal{T}\mathrm{op}}}(X \times |\Delta[n]|, Y)$;

\item the tensor of $X$ with $Z$ to be the space $X \times |Z|$; and

\item the cotensor of $X$ with $Z$ to be the space $X^{|Z|}$ of maps from $|Z|$ to $X$ with the compact-open topology.

\end{enumerate}
With these definitions, ${\mathcal{T}\mathrm{op}}$ has the structure of a simplicial model category.
\end{proposition}

The model category of topological spaces differs from that of simplicial sets in three noteworthy respects.

\begin{enumerate}

\item It is not combinatorial because the category ${\mathcal{T}\mathrm{op}}$ is not locally presentable \cite[Example 1.18(5)]{AdamekRosicky1994}.

\item It is left proper \cite[Theorem 13.1.10]{DwyerSpalinski1995}, but not by Proposition \ref{leftpropernoworries}, for not every space is cofibrant.

\item Every topological space is Serre fibrant, whereas not every simplicial set is a Kan complex.

\end{enumerate}

Nevertheless, the geometric realization-singular complex adjunction is a Quillen equivalence \cite[Chapter I.4, Example 2]{Quillen1967}, which renders simplicial sets a model for spaces. Still, the first difference raises the question of whether ${\mathcal{T}\mathrm{op}}$ can be localized. We note that ${\mathcal{T}\mathrm{op}}$ shares with ${\mathcal{SS}\mathrm{ets}}$ the feature of being a cellular model category \cite[Proposition 12.1.4]{Hirschhorn2003}. This property is a different way of strengthening cofibrant generation, and it lets us localize ${\mathcal{T}\mathrm{op}}$, for Theorem \ref{lblexists} holds if we replace ``combinatorial'' with ``cellular'' \cite[Theorem 4.1.1]{Hirschhorn2003}.

\subsection{The family of model structures on topological spaces}\label{section10.3}
Corollary \ref{illdothisonenonsc} and Proposition \ref{fiblocweqs} describe the left Bousfield localization of $\mathcal{T}\mathrm{op}$ at the set of maps $\{S^{n} \rightarrow S^{n}_{M^{-1}\mathbb{Z}} \mid n \geq 2\}$ with no additional work needed.

\begin{theorem}\label{ididitididitnsct} Let $M$ be a multiplicative subset of $\mathbb{Z}$. There is a model structure $\mathcal{L}_{M^{-1}\mathbb{Z}}{\mathcal{T}\mathrm{op}}$ on the category of topological spaces in which a map $f$ is

\begin{enumerate}

\item a weak equivalence if and only if $f$ is an $M$-local homotopy equivalence; and

\item a cofibration if and only if $f$ is a retract of a generalized relative CW inclusion.

\end{enumerate}
A space $X$ is fibrant in $\mathcal{L}_{M^{-1}\mathbb{Z}}{\mathcal{T}\mathrm{op}}$ if and only if $X$ is $M$-local. The model category $\mathcal{L}_{M^{-1}\mathbb{Z}}{\mathcal{T}\mathrm{op}}$ is cellular, left proper, and simplicial with the structure from Proposition \ref{topsimpv1}.
\end{theorem}
In particular, we obtain a model category whose weak equivalences are the rational homotopy equivalences and whose fibrant objects are the rational spaces.

\begin{remark}\label{ididitididitnsctrmk} By Proposition \ref{comdiagramheqs}, $\mathcal{L}_{M^{-1}\mathbb{Z}}{\mathcal{T}\mathrm{op}}$ is also the left Bousfield localization of ${\mathcal{T}\mathrm{op}}$ at the set of maps \[|T_{M^{-1}\mathbb{Z}}|=\{|\operatorname{Sing}(S^{n})| \rightarrow |\operatorname{Sing}(S^{n}_{M^{-1}\mathbb{Z}})| \mid n \geq 2\}.\] Thus, by \cite[Theorem 3.3.20]{Hirschhorn2003}, the Quillen equivalence $|-| \colon \mathcal{SS}\mathrm{ets} \rightleftarrows \mathcal{T}\mathrm{op} \colon \operatorname{Sing}$ induces a Quillen equivalence
\[|-| \colon \mathcal{L}_{M^{-1}\mathbb{Z}}\mathcal{SS}\mathrm{ets} \rightleftarrows \mathcal{L}_{M^{-1}\mathbb{Z}}\mathcal{T}\mathrm{op} \colon \operatorname{Sing}.\]
\end{remark}

\subsection{Rational homotopy versus rational homology}\label{section10.4}
We conclude this paper by contrasting our localization of all spaces at rational homotopy to Bousfield's localization at rational homology. Our argument is best presented using topological spaces, so we first formulate Theorem \ref{ptedhomloc} for topological spaces.

\begin{theorem}[{\cite[Theorem 10.2]{Bousfield1975}}]\label{ptedhomloctop} Let $M$ be a multiplicative subset of $\mathbb{Z}$. There is a model structure $\mathcal{L}_{HM^{-1}\mathbb{Z}}\mathcal{T}\mathrm{op}$ on the category of topological spaces in which a map $f$ is

\begin{enumerate}

\item a weak equivalence if and only if $f$ is an $M$-local homology isomorphism; and

\item a cofibration if and only if $f$ is a retract of a generalized relative CW inclusion.
\end{enumerate}
This model category is cellular, left proper, and simplicial with the structure from Proposition \ref{topsimpv1}.
\end{theorem}

\begin{remark}\label{ptedhomloctoprmk} By \cite[\S 2.5]{Bousfield1997}, $\mathcal{L}_{HM^{-1}\mathbb{Z}}\mathcal{T}\mathrm{op}$ is the left Bousfield localization of $\mathcal{T}\mathrm{op}$ at a single map, so it inherits the properties of being cellular, left proper, and simplicial.
\end{remark}
We show that our localization at rational homotopy $\mathcal{L}_{\mathbb{Q}}\mathcal{T}\mathrm{op}$ differs from that at rational homology $\mathcal{L}_{H\mathbb{Q}}\mathcal{T}\mathrm{op}$.

\begin{proposition}\label{theyarenotthesame} The model structures $\mathcal{L}_{H\mathbb{Q}}\mathcal{T}\mathrm{op}$ and $\mathcal{L}_{\mathbb{Q}}\mathcal{T}\mathrm{op}$ on topological spaces are not equal.
\end{proposition}

\begin{proof} We show that $c \colon \mathbb{R}P^{2} \rightarrow |\Delta[0]|$ is a weak equivalence in $\mathcal{L}_{H\mathbb{Q}}\mathcal{T}\mathrm{op}$ but not in $\mathcal{L}_{\mathbb{Q}}\mathcal{T}\mathrm{op}$. Firstly, $c$ is a weak equivalence in $\mathcal{L}_{H\mathbb{Q}}\mathcal{T}\mathrm{op}$ because, by Example \ref{defhlgyisodef2}, $c$ is a rational homology isomorphism. However, $c$ is not a weak equivalence in $\mathcal{L}_{\mathbb{Q}}\mathcal{T}\mathrm{op}$ because, by Example \ref{defrhtnscdef22}, $c$ is not a rational homotopy equivalence.
\end{proof}

\bibliographystyle{abbrv}
\bibliography{Eleftherios_Chatzitheodoridis_-_A_modern_perspective_on_rational_homotopy_theory_v2}

\begin{thebibliography}{10}

\bibitem{AdamekRosicky1994}
J.~Ad{\'{a}}mek and J.~Rosick{\'{y}}.
\newblock {\em Locally Presentable and Accessible Categories}.
\newblock Cambridge University Press, 1994.

\bibitem{Barwick2010}
C.~Barwick.
\newblock On left and right model categories and left and right {B}ousfield
  localizations.
\newblock {\em Homology Homotopy Appl.}, 12(2):245--320, 2010.

\bibitem{BastardasCasacuberta2001}
G.~Bastardas and C.~Casacuberta.
\newblock A homotopy idempotent construction by means of simplicial groups.
\newblock {\em Israel J. Math.}, 121(1):333--349, 2001.

\bibitem{BHKKRS2015}
M.~Bayeh, K.~Hess, V.~Karpova, M.~K{\k{e}}dziorek, E.~Riehl, and B.~Shipley.
\newblock Left-induced model structures and diagram categories.
\newblock In {\em Women in Topology: Collaborations in Homotopy Theory}, volume
  641 of {\em Contemp. Math.}, pages 49--82. American Mathematical Society,
  2015.

\bibitem{Bergner2018}
J.~E. Bergner.
\newblock {\em The Homotopy Theory of ({${\infty}$},1)-Categories}.
\newblock Cambridge University Press, 2018.

\bibitem{Bergner2022}
J.~E. Bergner.
\newblock Simplicial sets in topology, category theory, and beyond.
\newblock {\em Matematica}, 1(4):886--912, 2022.

\bibitem{Bousfield1975}
A.~Bousfield.
\newblock The localization of spaces with respect to homology.
\newblock {\em Topology}, 14(2):133--150, 1975.

\bibitem{Bousfield1977}
A.~Bousfield.
\newblock Constructions of factorization systems in categories.
\newblock {\em J. Pure Appl. Algebra}, 9(2):207--220, 1977.

\bibitem{Bousfield1997}
A.~Bousfield.
\newblock Homotopical localization of spaces.
\newblock {\em Amer. J. Math.}, 119(6):1321--1354, 1997.

\bibitem{BousfieldGugenheim1976}
A.~Bousfield and V.~Gugenheim.
\newblock {\em On PL De Rham Theory and Rational Homotopy Type}, volume~8 of
  {\em Memoirs of the American Mathematical Society}.
\newblock American Mathematical Society, 1976.

\bibitem{BousfieldKan1972}
A.~Bousfield and D.~Kan.
\newblock {\em Homotopy Limits, Completions and Localizations}.
\newblock Springer-Verlag Berlin Heidelberg, 1972.

\bibitem{Burke2021}
N.~Burke.
\newblock {\em Homotopy Theory of Monoids and Group Completion}.
\newblock PhD thesis, Apollo - University of Cambridge Repository, 2021.

\bibitem{Casacuberta1993}
C.~Casacuberta.
\newblock On the rationalization of the circle.
\newblock {\em Proc. Amer. Math. Soc.}, 118(3):995--1000, 1993.

\bibitem{CasacubertaPeschke1993}
C.~Casacuberta and G.~Peschke.
\newblock Localizing with respect to self-maps of the circle.
\newblock {\em Trans. Amer. Math. Soc.}, 339(1):117--140, 1993.

\bibitem{CPP1992}
C.~Casacuberta, G.~Peschke, and M.~Pfenniger.
\newblock On orthogonal pairs in categories and localization.
\newblock In {\em Adams Memorial Symposium on Algebraic Topology, Volume 1},
  London Mathematical Society Lecture Note Series, pages 211--224. Cambridge
  University Press, 1992.

\bibitem{DavisKirk2001}
J.~F. Davis and P.~Kirk.
\newblock {\em Lecture Notes in Algebraic Topology}.
\newblock American Mathematical Society, 2001.

\bibitem{Dugger2001}
D.~Dugger.
\newblock Combinatorial model categories have presentations.
\newblock {\em Adv. Math.}, 164(1):177--201, 2001.

\bibitem{DwyerSpalinski1995}
W.~G. Dwyer and J.~Spalinski.
\newblock Homotopy theories and model categories.
\newblock In {\em Handbook of Algebraic Topology}, pages 1--56. Elsevier
  Science B.V., 1995.

\bibitem{Farjoun1992}
E.~D. Farjoun.
\newblock Homotopy localization and {$v_{1}$}-periodic spaces.
\newblock In {\em Algebraic Topology: Homotopy and Group Cohomology}, pages
  104--113. Springer-Verlag Berlin Heidelberg, 1992.

\bibitem{Farjoun1995_CI}
E.~D. {Farjoun}.
\newblock Cellular inequalities.
\newblock In {\em The {\v C}ech Centennial}, Contemp. Math., pages 159--181.
  American Mathematical Society, 1995.

\bibitem{Farjoun1995}
E.~D. Farjoun.
\newblock {\em Cellular Spaces, Null Spaces and Homotopy Localization}.
\newblock Springer Science and Business Media, 1995.

\bibitem{FHT2001}
Y.~F{\'e}lix, S.~Halperin, and J.-C. Thomas.
\newblock {\em Rational Homotopy Theory}.
\newblock Springer-Verlag New York, 2001.

\bibitem{GoerssJardine1999}
P.~G. Goerss and J.~F. Jardine.
\newblock {\em Simplicial Homotopy Theory}.
\newblock Birkh{\"a}user Verlag, 1999.

\bibitem{GTHT2000}
A.~G{\'{o}}mez-Tato, S.~Halperin, and D.~Tanr{\'{e}}.
\newblock Rational homotopy theory for non-simply connected spaces.
\newblock {\em Trans. Amer. Math. Soc.}, 352(4):1493--1525, 2000.

\bibitem{HMR1975}
P.~Hilton, G.~Mislin, and J.~Roitberg.
\newblock {\em Localization of Nilpotent Groups and Spaces}.
\newblock North-Holland Publishing Company, 1975.

\bibitem{Hirschhorn2003}
P.~S. Hirschhorn.
\newblock {\em Model Categories and Their Localizations}.
\newblock American Mathematical Society, 2003.

\bibitem{Ivanov2022}
S.~O. Ivanov.
\newblock An overview of rationalization theories of non-simply connected
  spaces and non-nilpotent groups.
\newblock {\em Acta Math. Sin. (Engl. Ser.)}, 38(10):1705--1721, 2022.

\bibitem{Llerena1982}
I.~Llerena.
\newblock Localization of nilpotent fibre maps.
\newblock {\em Collect. Math.}, 33(2):177--185, 1982.

\bibitem{Llerena1985}
I.~Llerena.
\newblock Localization of fibrations with nilpotent fibre.
\newblock {\em Math. Z.}, 188:397--410, 1985.

\bibitem{May1967}
J.~P. May.
\newblock {\em Simplicial Objects in Algebraic Topology}.
\newblock The University of Chicago Press, 1967.

\bibitem{Neisendorfer1978}
J.~Neisendorfer.
\newblock Lie algebras, coalgebras and rational homotopy theory for nilpotent
  spaces.
\newblock {\em Pacific J. Math.}, 74(2):429--460, 1978.

\bibitem{Quillen1967}
D.~G. Quillen.
\newblock {\em Homotopical Algebra}.
\newblock Springer-Verlag Berlin Heidelberg, 1967.

\bibitem{Quillen1969}
D.~G. Quillen.
\newblock Rational homotopy theory.
\newblock {\em Ann. of Math.}, 90(2):205--295, 1969.

\bibitem{RWZ2021}
M.~Rivera, F.~Wierstra, and M.~Zeinalian.
\newblock Rational homotopy equivalences and singular chains.
\newblock {\em Algebr. Geom. Topol.}, 21(3):1535--1552, 2021.

\bibitem{Rotman2009}
J.~J. Rotman.
\newblock {\em An Introduction to Homological Algebra}.
\newblock Springer Science and Business Media, LLC, 2nd edition, 2009.

\bibitem{Serre1951}
J.-P. Serre.
\newblock Homologie singuli{\`e}re des espaces fibr{\'e}s.
\newblock {\em Ann. of Math.}, 54(3):425--505, 1951.

\bibitem{Sullivan1974}
D.~Sullivan.
\newblock Genetics of homotopy theory and the {A}dams conjecture.
\newblock {\em Ann. of Math.}, 100(1):1--79, 1974.

\bibitem{Sullivan1977}
D.~Sullivan.
\newblock Infinitesimal computations in topology.
\newblock {\em Publ. Math. Inst. Hautes {\'E}tudes Sci.}, 47:269--331, 1977.

\bibitem{Sullivan1970}
D.~P. Sullivan.
\newblock {\em Geometric Topology: Localization, Periodicity and Galois
  Symmetry}.
\newblock Springer, 2005.

\end{thebibliography}

\end{document}